\documentclass[12pt, reqno]{amsart}
\makeatletter
\DeclareMathAlphabet{\mathpzc}{OT1}{pzc}{m}{it}
\usepackage[margin=2cm]{geometry}
\usepackage{amsmath, amssymb, amsfonts, amstext, verbatim, amsthm, mathrsfs}
\usepackage[mathcal]{eucal}
\usepackage{microtype}
\usepackage[all,arc,knot,poly]{xy}
\usepackage{enumerate}
\usepackage{mathrsfs}
\usepackage{mathtools}
\usepackage{amsxtra}
\usepackage{needspace}
\usepackage{graphicx}
\usepackage{setspace}

\usepackage[usenames]{color}
\usepackage{aliascnt} 
\usepackage[colorlinks=true,linkcolor=blue,citecolor=blue,urlcolor=blue,citebordercolor={0 0 1},urlbordercolor={0 0 1},linkbordercolor={0 0 1}]{hyperref} 
\usepackage{enumerate}
\usepackage{xspace}

\usepackage{tikz}
\usetikzlibrary{decorations.pathreplacing,cd}
\usetikzlibrary{matrix,arrows}

\usepackage{xy}

\theoremstyle{plain}

\newcommand{\refnewtheoremn}[4]{
\newaliascnt{#1}{#2}
\newtheorem{#1}[#1]{#3}
\aliascntresetthe{#1}
\expandafter\providecommand\csname #1autorefname\endcsname{#4}}

\newcommand{\refnewtheorem}[3]{\refnewtheoremn{#1}{#2}{#3}{#3}}

\def\makeCal#1{
\expandafter\newcommand\csname c#1\endcsname{\mathcal{#1}}}
\def\makeBB#1{
\expandafter\newcommand\csname b#1\endcsname{\mathbb{#1}}}
\def\makeFrak#1{
\expandafter\newcommand\csname f#1\endcsname{\mathfrak{#1}}}

\count@=0
\loop
\advance\count@ 1
\edef\y{\@Alph\count@}
\expandafter\makeCal\y
\expandafter\makeBB\y
\expandafter\makeFrak\y
\ifnum\count@<26
\repeat

\newtheorem{thm}{Theorem}[section]
\newtheorem{theorem}{Theorem}[section]

\refnewtheorem{lemma}{thm}{Lemma}
\refnewtheorem{claim}{thm}{Claim}
\refnewtheorem{cor}{thm}{Corollary}
\refnewtheorem{conj}{thm}{Conjecture}
\refnewtheorem{prop}{thm}{Proposition}
\refnewtheorem{proposition}{thm}{Proposition}
\refnewtheorem{quest}{thm}{Question}

\refnewtheorem{assumption}{thm}{Assumption}
\refnewtheorem{problem}{thm}{Problem}

\theoremstyle{definition}
\refnewtheorem{remark}{thm}{Remark}
\refnewtheorem{remarks}{thm}{Remarks}
\refnewtheorem{defn}{thm}{Definition}
\refnewtheorem{definition}{thm}{Definition}
\refnewtheorem{definitions}{thm}{Definitions}

\refnewtheorem{notn}{thm}{Notation}
\refnewtheorem{const}{thm}{Construction}
\refnewtheorem{example}{thm}{Example}
\refnewtheorem{fact}{thm}{Fact}

\newcommand {\id}{\operatorname{id}}

\renewcommand{\Im}{\operatorname{Im}}
\renewcommand{\Re}{\operatorname{Re}}
\newcommand {\Hom}{\operatorname{Hom}}
\newcommand {\Aut}{\operatorname{Aut}}

\newcommand{\DT}{\operatorname{DT}}

\newcommand{\Ad}{\operatorname{Ad}}

\newcommand{\into}{\hookrightarrow}

\newcommand {\<}{\langle}
\renewcommand {\>}{\rangle}
\newcommand {\bra}{\langle}
\newcommand {\ket}{\rangle}
\newcommand{\isom}{\cong}
\newcommand{\half}{\frac{1}{2}}
\newcommand{\tensor}{\otimes}

\newcommand{\Mer}{\mathcal{M}}

\newcommand{\del}{\partial}

\newcommand{\dual}{\vee}

\renewcommand{\b}{\, \b \,}

\renewcommand{\b}{\, | \,}
\newcommand{\wb}{\, \Big|\, }

\renewcommand{\omega}{a}

\begin{document}

\title[A quantized Riemann-Hilbert problem in Donaldson-Thomas theory]{A quantized Riemann-Hilbert problem in\\ Donaldson-Thomas theory}
\author{Anna Barbieri, Tom Bridgeland and  Jacopo Stoppa}

\address{Universit\`a degli Studi di Milano, Dipartimento di matematica \lq\lq F.\ Enriques\rq\rq ,
Via C.\ Saldini 50, 20133 Milano, Italy \newline\indent
(previously) University of Sheffield,  Hicks building, Hounsfield Road, S3 7RH Sheffield, UK
}
\email{anna.barbieri1@unimi.it}
\address{University of Sheffield, Hicks building, Hounsfield Road, S3 7RH Sheffield, UK}
\email{t.bridgeland@sheffield.ac.uk}
\address{SISSA, via Bonomea 265, 34136 Trieste, Italy, and \newline\indent Institute for Geometry and Physics, Strada Costiera 11, 34151 Trieste, Italy}
\email{jstoppa@sissa.it}

\begin{abstract}
	We introduce  Riemann-Hilbert problems determined by refined Donaldson-Thomas theory. They involve piecewise holomorphic maps from the complex plane to the group of automorphisms of a quantum torus algebra. We study the simplest case in detail and use the Barnes double gamma function to construct a solution.
\end{abstract}

\maketitle

\section{Introduction}

There has been  recent interest in a class of Riemann-Hilbert problems  that are naturally suggested by the form of the wall-crossing formula in Donaldson-Thomas (DT) theory. These problems involve piecewise holomorphic maps from the complex plane to the group of automorphisms of a Poisson algebraic torus, with  discontinuities along a  collection of rays prescribed by the DT invariants. Such problems appeared in the physics literature in the  work of Gaiotto, Moore and Neitzke \cite{GMN1,GMN2}, and have since been considered  by mathematicians \cite{RHuncoupled,RHDT,conifold,FS}. 

The works listed above are mostly concerned with  Riemann-Hilbert problems  defined using unrefined DT invariants. In this paper we consider an analogous class of  Riemann-Hilbert problems arising in refined DT theory. These  involve maps into the group of automorphisms of  a quantum torus algebra. Earlier  discussions of such  quantum Riemann-Hilbert problems appear  in \cite{CNV,FS2}.

In this paper we consider the special case of a refined BPS structure satisfying the conditions of Definition \ref{leaf} below. The basic example  is the one arising from the refined DT theory of  the A$_1$ quiver.  We give an explicit solution to the corresponding quantum Riemann-Hilbert problem in terms of products of modified gamma functions. We also write the  solution  in adjoint form using a modified version of the Barnes double gamma function.
It is intriguing to note that this same function arises in expressions for the partition functions of supersymmetric gauge theories \cite[Appendix A]{ON}.

We conclude by discussing two natural limits of the adjoint form of the  solution, which both relate to the classical Riemann-Hilbert problem studied in \cite{RHuncoupled,RHDT}. In one of these limits we find the Hamiltonian generating function for the classical solution. In the other, we rather unexpectedly find the  $\tau$-function introduced in \cite{RHDT}. 

 \subsection{Refined BPS structures}
 
 In \cite{RHDT} the output of unrefined DT theory was axiomatised to give the definition of a BPS structure. This is a special case of Kontsevich and Soibelman's notion of a stability structure \cite{KS}. The natural analogue for refined DT theory  reads as follows (compare  also \cite[Section 4]{FS2}).
 
 \begin{defn}
 A refined BPS structure $(\Gamma,Z,\Omega)$ consists of data\begin{itemize}
\item[(a)] A finite-rank free abelian group $\Gamma\isom \bZ^{\oplus n}$, equipped with a skew-symmetric form \[\<-,-\>\colon \Gamma \times \Gamma \to \bZ;\]

\item[(b)] A homomorphism of abelian groups
$Z\colon \Gamma\to \bC$;\smallskip

\item[(c)] A map of sets
\[\Omega\colon \Gamma\to \bQ\big[\bL^{\pm \half}\big], \qquad \Omega(\gamma)=\sum_{n\in \bZ} \Omega_n(\gamma)\cdot (\bL^{\half})^n,\]
\end{itemize}

where $\bL^{\frac{1}{2}}$ is a formal symbol, satisfying the following two conditions:
\begin{itemize}
\item[(i)] Symmetry: $\Omega(-\gamma) = \Omega(\gamma)$ for all $\gamma \in \Gamma$, and $\Omega(0) = 0$;\smallskip

\item[(ii)] Support property: fixing a norm  $\|\cdot\|$ on the finite-dimensional vector space $\Gamma\tensor_\bZ \bR$, there is a constant $C>0$ such that 
\begin{equation*}
\Omega(\gamma)\neq 0 \implies |Z(\gamma)|> C\cdot \|\gamma\|.\end{equation*}
\end{itemize}
\end{defn}
The original definition is recovered by setting $\bL=1$ and taking $\Omega$ to be a map $\Gamma\to \bQ$.

The support property will  in fact play no role in what follows since we  only consider refined BPS structures satisfying much stronger finiteness constraints. 

\begin{defn}
\label{leaf}
We say that a refined BPS structure $(Z,\Gamma,\Omega)$ is
\begin{itemize}
\item[(a)] finite if $\Omega(\gamma)=0$ for all but finitely many classes $\gamma\in \Gamma$;
\smallskip

\item[(b)] uncoupled if
$\Omega(\gamma_1)\neq 0 \text{ and } \Omega(\gamma_2)\neq 0 \implies \<\gamma_1,\gamma_2\>=0;$
\smallskip

\item[(c)] palindromic if  $\Omega_n(\gamma)=\Omega_{-n}(\gamma)$ for all $n\in \bZ$ and $\gamma\in \Gamma$;\smallskip
\item[(d)] integral if $\Omega_n(\gamma)\in \bZ$ for all $n\in \bZ$ and $\gamma\in \Gamma$.\end{itemize}
\end{defn}

For the most part in this paper we shall restrict attention to the  following example, which  satisfies all the  conditions of Definition \ref{leaf}.
\begin{example}[Doubled A$_1$ structure]
\label{mineral}
Given an element $z\in \bC^*$ there is an associated refined BPS structure $(\Gamma,Z,\Omega)$  defined by the following   data:
\begin{itemize}
\item[(a)] the lattice $\Gamma\isom \bZ^{\oplus 2}$ and skew-symmetric form $\<-,-\>$ are
\[\Gamma=\bZ\alpha\oplus\bZ\alpha^\vee, \qquad \<\alpha^\vee,\alpha\>=1;\]

\item[(b)] the group homomorphism $Z\colon \Gamma\to \bC$ is determined by \[z=Z(\alpha)\in \bC^*,\qquad Z(\alpha^\vee)=0;\]

\item[(c)] the only nonzero refined BPS invariants are $\Omega(\pm \alpha)=1$.
\end{itemize}
This example  corresponds mathematically to the refined Donaldson-Thomas theory of the A$_1$ quiver. In physical terms it describes the BPS spectrum of $U(1)$  gauge theory (see Remark \ref{physics}). We will refer to it as the doubled A$_1$ refined BPS structure.\end{example}

\subsection{The quantum Riemann-Hilbert problem}

Let $(\Gamma,Z,\Omega)$ be a refined BPS structure. 
Precisely formulating the associated quantum Riemann-Hilbert problem  is a non-trivial task, at least as difficult as the analogous problem in the  unrefined situation, which was discussed at length in \cite{RHDT}.
Here we just give the rough idea. 
Morally-speaking, the quantum Riemann-Hilbert problem  takes values in the group of automorphisms of the quantum torus algebra
\[\bC_q[\bT]=\bigoplus_{\gamma\in \Gamma} \bC\big[\bL^{\pm \frac{1}{2}}\big]\cdot x_\gamma,\qquad x_{\gamma_1}*x_{\gamma_2}= \bL^{\half\bra \gamma_1,\gamma_2\ket}  \cdot x_{\gamma_1+\gamma_2},\]
which is a quantization of the ring of functions on the algebraic torus \begin{equation}
\label{toto}\bT=\Hom_\bZ(\Gamma,\bC^*)\isom (\bC^*)^n.\end{equation}
We introduce the one-parameter family of automorphisms
\[\epsilon_Z(t)\in  \Aut \bC_q[\bT], \qquad \epsilon_Z(t)(x_\gamma)=\exp(Z(\gamma)/t)\cdot x_\gamma,\]
which lift the pull-backs by the corresponding translations of the classical torus.

Recall the quantum dilogarithm function\footnote{The reader should be aware that there are different conventions for this function in the literature (see 
Remark \ref{convent}).}\begin{equation}
\label{qexp}\bE_q(x)=\prod_{k\geq 0} (1-q^k x),\qquad \bE_q(x)^{-1}=\sum_{n\geq 0} \frac{ x^n }{(1-q)\cdots (1-q^n)}.\end{equation}
The active rays $\ell\subset \bC^*$  of the BPS structure are defined to be the rays of the form $\ell=\bR_{>0}\cdot Z(\gamma)$ for classes  $\gamma\in \Gamma$ satisfying $\Omega(\gamma)\neq 0$. To each such ray  we would like to attach a product\footnote{\label{change}The
formula  for $\DT_q(\ell)$ given in the published version is incorrect, although since it gives the correct answer in Example \ref{mineral}  the only further changes required are to Sections 5.2 and 5.3. We thank Sergey Alexandrov for pointing out the error and patiently helping us to fix it.}
\begin{equation}\label{tir}\DT_q(\ell)=\prod_{Z(\gamma)\in \ell}\,  \prod_{n\in \bZ} \, \bE_\bL\left((- \bL^\half)^{n+1} x_\gamma\right)^{(-1)^{n-1}\,\Omega_n(\gamma)}\in \bC_q[\bT],\end{equation}
and then consider the associated automorphism
\begin{equation}
\label{end}\bS_q(\ell)=\Ad_{\DT_q(\ell)} \in  \Aut \bC_q[\bT].\end{equation}

Of course, since the product in \eqref{qexp} is infinite, the expressions \eqref{tir} and \eqref{end}  do not make rigorous sense, and it is therefore necessary to work in some extension of the quantum torus, or perhaps to pursue some entirely different approach.
Ignoring these difficulties for a moment longer, the quantum Riemann-Hilbert problem can be  roughly stated as follows.

\begin{problem}\label{rough} (Heuristic version). Find a piecewise holomorphic function
\[\Phi \colon \bC^* \to  \Aut \bC_q[\bT],\]
satisfying the three conditions
\begin{itemize}
\item[(i)] As $t$ crosses an active ray $\ell\subset \bC^*$ in the clockwise direction, the map $\Phi(t)$ jumps by the corresponding automorphism
\[\Phi(t)\mapsto \Phi(t)\circ \bS_q(\ell);\]
	
	\item[(ii)] As $t\to 0$  we have $\Phi(t) \circ \epsilon_Z(t)\to \id$;
\smallskip

\item[(iii)] The function $\Phi(t)$ has \lq\lq moderate growth\rq\rq\ as $t\to \infty$.
\end{itemize}
\end{problem}

At present we only know how to  rigorously formulate, let alone solve, the quantum Riemann-Hilbert problem  under the  conditions (a)-(d) of Definition \ref{leaf}, and we will therefore restrict to this case from now on.
In fact, for any such refined BPS structure, the solution of the associated quantum Riemann-Hilbert problem can be obtained by  superimposing  solutions to  the  problem  defined by  the refined BPS structure of Example \ref{mineral}. Thus we will now restrict attention  to this doubled A$_1$ case. We will return to the general case  in Section \ref{general}.

\subsection{The doubled A$_1$  example}

Fix an element $z\in \bC^*$ and consider the corresponding refined BPS structure $(\Gamma,Z,\Omega)$ of Example \ref{mineral}. Introduce the following alternative generators of the quantum torus algebra $\bC_q[\bT]$:
\[q^{\frac{1}{2}}:=-\bL^{\frac{1}{2}}, \qquad y_{m\alpha+n\alpha^\vee}:=(-1)^{m(n+1)} \cdot x_{m\alpha+n\alpha^\vee}.\]For an explanation of these signs see Section \ref{quad}.  Although they are not strictly necessary, introducing them  will lead to simpler formulae below.  

There are two active rays  $\ell_{\pm} =\pm \bR_{>0}\cdot z$, and the corresponding expressions \eqref{tir} are
\[\DT_q(\ell_\pm)=\bE_{\bL}(-\bL^{\half} x_{\pm \alpha})^{-1}=\bE_q(-q^{\half} y_{\pm \alpha})^{-1}.\]
To make rigorous sense of these elements we will first  need to modify the quantum torus algebra.
Define the extended quantum torus algebra to be the non-commutative algebra
\begin{equation}\label{ex}\widehat{\bC_q[\bT]}=\bigoplus_{n\in \bZ} \Mer(\cH\times \bC) \cdot y_{n\alpha^\dual},\end{equation}
where $\Mer(\cH\times \bC)$ denotes the field of meromorphic functions $f(\tau,\theta)$ on the product of the upper half-plane $\cH$ with the complex plane $\bC$, and 
the product is
\begin{equation}\label{produ}\Big(f_1(\tau,\theta)\cdot y_{n_1\alpha^\dual}\Big) *\Big( f_2(\tau,\theta)\cdot y_{n_2\alpha^\dual}\Big)={ f_1(\tau,\theta)\cdot f_2(\tau,\theta+n_1\tau)}\cdot y_{(n_1+n_2)\alpha^\dual}.\end{equation}

There is a  commutative subalgebra
\begin{equation}\label{comm}\widehat{\bC_q[\bT]}_0=\Mer(\cH\times \bC)\cdot 1 \subset \widehat{\bC_q[\bT]}.\end{equation}
and  an injective homomorphism
$I\colon \bC_q[\bT]\into \widehat{\bC_q[\bT]}$ defined by \[I\colon q^{\frac{k}{2}} \cdot y_{m\alpha+n\alpha^\dual}\mapsto \exp\big( \pi i  (k+mn)\tau+2\pi i m\theta\big)\cdot y_{n\alpha^\dual}.\]
We will often identify elements of $\bC_q[\bT]$ with their images under the embedding $I$. Note that
\[I(q^{\half})=\exp(\pi i \tau)\cdot 1, \qquad I(y_\alpha)=\exp(2\pi i \theta)\cdot 1, \qquad I(y_{\alpha^{\vee}})=y_{\alpha^{\vee}}.\]

The product in the definition of the quantum dilogarithm \eqref{qexp} is absolutely convergent for $|q|<1$ and there are  therefore well-defined elements
\[\DT_q(\ell_\pm)=\bE_{e^{2\pi i \tau}}\big(\!  -\! e^{\pi i \tau\pm 2\pi i \theta}\big)^{-1}\cdot 1\in \widehat{\bC_q[\bT]}_0,\]
and corresponding automorphisms \[\bS_q(\ell_\pm)=\Ad_{\DT_q(\ell_\pm)}\in \Aut \widehat{\bC_q[\bT]}.\]

It is now straightforward to write down a rigorous version of the quantum Riemann-Hilbert problem (see Problem \ref{probl_trasl} below). Since there are only two active rays, analytically continuing the two parts of the solution leads to maps\[\Phi_\pm \colon \bC^*\setminus  i\ell_{\pm}   \to  \Aut \widehat{\bC_q[\bT]},\]
satisfying the jumping relation
	\begin{equation}
	\Phi_+(t)=\begin{cases}  \Phi_-(t) \circ \bS_q(\ell_+)&\mbox{ if }\Re(t/z)>0,\\  \Phi_-(t) \circ\bS_q(\ell_-)^{-1} &\mbox{ if }\Re(t/z)<0,\end{cases}
	\end{equation}
	together with natural limiting conditions at $t=0$ and $t=\infty$.

\subsection{Solution to the doubled A$_1$ problem}

Due to a symmetry of the problem, a solution to the above quantum Riemann-Hilbert problem can  be deduced from the solution to the corresponding classical  problem  studied in \cite{RHuncoupled}. Unfortunately we can only prove uniqueness properties for this solution under some strong additional hypotheses (see Remark \ref{uniquness}).

To describe this solution, let us first introduce
the equivalent maps
\[\Psi_\pm \colon \bC^*\setminus i\ell_{\pm} \to  \Aut \widehat{\bC_q[\bT]},\qquad \Psi(t)=\Phi(t)\circ \epsilon_Z(t). \]
We also introduce the 
modified  gamma function
	\begin{equation}\label{al}
	\Lambda(w,\eta\b 1):=\frac{\Gamma(w+\eta)\cdot e^w}{\sqrt{2\pi}\cdot w^{\eta+w-\frac{1}{2}}},
\end{equation}
which is meromorphic and single-valued for $w\in\bC^*\setminus\bR_{<0}$ and  $\eta\in\bC$.

\begin{thm}
\label{t2}
For each $z\in \bC^*$, the unique automorphisms  $\Psi_{\pm}(t)\in \Aut \widehat{\bC_q[\bT]}$ which act trivially on the subalgebra \eqref{comm},  and satisfy
\[\Psi_\pm (t)\left(y_{\alpha^{\vee}}\right) = \Lambda\left(\pm \frac{z}{2\pi i t},\frac{1}{2} \mp\Big(\theta+\frac{\tau}{2}\Big) \wb  1\right)^{\pm 1}\cdot y_{\alpha^{\vee}}, \]
give a  solution to the above quantum Riemann-Hilbert problem.\end{thm}

An interesting new possibility in the quantum case is to express the solution of Theorem \ref{t2} in adjoint form. Namely, for each $z\in \bC^*$,  we can write
\begin{equation}
\label{rainy}\Psi_\pm(t)=\Ad_{\psi_\pm(t)}, \qquad \psi_\pm\colon \bC^*\setminus  i\ell_{\pm}   \to  \widehat{\bC_q[\bT]_0}.\end{equation}
Although this does not specify the functions $\psi_\pm$ uniquely, we show that a natural choice   is to take
\begin{equation}
	\label{adjy}\psi_{\pm}(t)=F\Big(\pm \frac{z}{2\pi i t},\frac{1+\tau}{2}\mp \theta\wb1, \tau\Big)^{-1}\cdot 1,\end{equation}
where the function $F(w,\eta\b \omega_1,\omega_2)$ is a modification of the Barnes double gamma function. More precisely
\begin{equation*}
\label{f}F(w,\eta\b \omega_1,\omega_2)=\Gamma_2(w+\eta \b \omega_1,\omega_2)\cdot e^{g(w,\eta\b\omega_1,\omega_2)}\cdot w^{\frac{1}{2}B_{2,2}(w+\eta\b\omega_1,\omega_2)},\end{equation*}
where $\Gamma_2$ denotes the Barnes double gamma function, and
\[g(w,\eta\b\omega_1,\omega_2)={-\frac{3w^2}{4\omega_1\omega_2}-\frac{\eta w}{\omega_1\omega_2}+\frac{w(\omega_1+\omega_2)}{2\omega_1\omega_2}},\]
\[B_{2,2}(x\b\omega_1,\omega_2)=\frac{x^2}{\omega_1\omega_2} -\Big(\frac{1}{\omega_1} + \frac{1}{\omega_2}\Big) x+\frac{1 }{6}\Big(\frac{\omega_2}{\omega_1} + \frac{\omega_1}{\omega_2}\Big)+\frac{1}{2}.\]

This modification  of the double gamma function is entirely analogous to the modification \eqref{al} of the usual gamma function: it is designed to eliminate the  sporadic terms in the large $|w|$ asymptotic expansion of the function $\Gamma_2(w+\eta\b \omega_1,\omega_2)$. In fact we prove that when $\Re(\omega_i)>0$ there is an expansion \[\log F(w,\eta \b \omega_1,\omega_2) \sim
\sum_{k\geq 1}\frac{(-1)^k\cdot B_{2,k+2}(\eta\b \omega_1,\omega_2)}{k(k+1)(k+2)} \cdot w^{-k},
	\]
where $B_{2,n}(x\b \omega_1,\omega_2)$ denote the double Bernoulli polynomials, which is valid as $|w|\to \infty$ in any closed subsector of the half-plane $\Re(w)>0$.

\subsection{Two interesting  limits} There are two limits to the solution of Theorem \ref{t2} which it is interesting to consider, and which relate to the  classical Riemann-Hilbert problem  studied in \cite{RHuncoupled,RHDT}. The classical problem, at least heuristically,  looks for  maps
\[\Phi_\pm \colon \bC\setminus i\ell_{\pm}\to \Aut(\bT),\]
where $\bT$ is the classical torus \eqref{toto}, satisfying conditions analogous to those in Problem \ref{rough} above. As above, it is convenient to write
\[\Psi_\pm \colon \bC\setminus i\ell_{\pm}\to \Aut(\bT),\qquad \Phi_{\pm}(t)= \Psi_{\pm}(t)\circ \epsilon_Z(t),\]
where $\epsilon_Z(t)$ is the translation of $\bT$ defined by \[\epsilon_Z(t)^*(y_\gamma)=\exp(Z(\gamma)/t)\cdot y_\gamma.\]
 It is shown in \cite{RHuncoupled} that a possible solution  is given by
\begin{equation}
	\label{cl2}\Psi_\pm(t)^*(y_{\alpha})=y_\alpha, \qquad \Psi_\pm(t)^*(y_{\alpha^\vee})= \Lambda\left(\pm \frac{z}{2\pi i t}, \frac{1}{2}\mp \theta\wb1\right)^{\pm 1}\cdot y_{\alpha^\vee},\end{equation}
where we have written $y_\alpha=\exp(2\pi i \theta)$.

The first limit  consists of sending $\tau\to 0$. The product on the quantum torus $\bC_q[\bT]$ induces in this limit a  Poisson bracket on  $\bT$ given explicitly by
\[\{y_{\gamma_1},y_{\gamma_2}\}=\<\gamma_1,\gamma_2\> \cdot y_{\gamma_1+\gamma_2}.\]
In this limit  the solution to the quantum Riemann-Hilbert problem specified in Theorem \ref{t2} becomes the solution \eqref{cl2} to the corresponding classical problem. The adjoint description \eqref{adjy} becomes the statement  that the automorphisms $\Psi_\pm(t)\in \Aut(\bT)$  which solve the classical problem are the time 1 Hamiltonian flow of the functions
\[H_\pm (z,t,\theta)=\lim_{\tau\to 0} \Big((2\pi i \tau)\cdot \log \psi_{\pm}(t)\Big).\]
In terms of the explicit description \eqref{cl2} this is   the identity
\[\frac{\partial }{\partial \theta} H_\pm (z,t,\theta)= \mp (2\pi i)\cdot \log \Lambda\Big(\pm \frac{z}{2\pi i t}, \frac{1}{2}\mp \theta\b1\Big).\]
In Section \ref{abab} we give an explicit description of the functions $H_\pm(z,t,\theta)$ in terms of the Barnes $G$-function.

The second limit consists of setting $\tau=1$ and hence $q^{\half}=-1$, and seems to correspond in physical terms to the Nekrasov limit. Although the quantum torus algebra $\bC_q[\bT]$ becomes commutative in this limit, the extension \eqref{ex}-\eqref{produ} does not. It is convenient to express the limiting  function in the form
\begin{equation}
	\label{finished}\Upsilon\left(\frac{\pm z}{2\pi i t},\mp \theta\right)=  \left( \frac{\pm z}{2\pi i t}\right)^{\frac{1}{12}}\cdot \lim_{\tau\to 1} \psi_\pm(t),\end{equation}where the function $\Upsilon(w,\eta)$  can again be expressed in terms of the Barnes G-function \eqref{upsy}.

 When $\theta=0$
the expressions \eqref{finished} coincide with the $\tau$-functions appearing in \cite[Section 5.4]{RHDT}, so we can view \eqref{finished} as an extension of this function to arbitrary values of $\theta$. Note however that there is a confusing shift  here: with our conventions the classical Riemann-Hilbert problem studied in \cite[Section 5.3]{RHDT} corresponds to $\theta=\half$.
The defining relation \eqref{rainy} gives in the limit an identity
\[\Upsilon\left(\pm \frac{ z}{2\pi i t},\mp\left(\theta+\half\right)\right)=\Upsilon\left(\pm \frac{ z}{2\pi i t}\mp\left(\theta-\half\right)\right)\cdot\Lambda\left(\pm \frac{z}{2\pi i t}, \half\mp \theta\b1\right)^{\pm 1}.\]
This difference relation may give some clue  as to the true nature of the $\tau$-function, whose definition  in \cite{RHDT} remains rather mysterious.

\subsection*{Acknowledgments} We thank Dylan Allegretti,  Pierrick Bousseau, Lotte Hollands, Sven Meinhardt, Andy Neitzke, Tom Sutherland, and particularly John Calabrese, for useful comments and correspondence. The first two authors have received funding from the European Research Council, ERC-AdG  StabilityDTCluster. 


\section{Special functions}
\label{sec_special_funct}

The solution to our quantum Riemann-Hilbert problem \label{prob_pa} can be expressed using modified versions of the Barnes multiple gamma functions $\Gamma_1(x\b\omega_1)$ and $\Gamma_2(x\b\omega_1,\omega_2)$. In this section we recall the definition of these  functions and  review some of their basic properties. We then introduce the two modifications $\Lambda(w,\eta\b\omega_1)$ and $F(w,\eta\b\omega_1,\omega_2)$ appearing in the Introduction. 

\subsection{Multiple Bernoulli polynomials}
\label{sec_B2}
Let $N>0$ be a positive integer, and  fix a vector of non-zero complex numbers \begin{equation*}
\underline{a}=(a_1,\dots, a_N)\in (\bC^*)^N.\end{equation*}
In what follows we shall make frequent use  of the  multiple Bernoulli polynomials $B_{N,k}(x\b \underline{a})$. These polynomials are defined by the expansion
	\begin{equation}\label{Bpoly}
	\frac{t^N e^{xt}}{\prod_{i=1}^N(e^{a_i t}-1)} =\sum_{k\geq 0}  B_{N,k}(x \b \underline a)\cdot \frac{t^k}{k!}.
	\end{equation}
	They satisfy the difference relations
\[B_{N,k}(x+a_i\b a_1,\cdots a_N) - B_{N,k}(x\b a_1,\cdots a_N) = k\,B_{N-1,k-1}(x\b a_1,\cdots, a_{i-1},a_{i+1},\cdots,a_N),\] 
and the homogeneity property 
\begin{equation}
\label{homoBnk}B_{N,k}(\lambda x\b \lambda\underline a) = \lambda^{k-N} B_{N,k}(x\b \underline{a}), \qquad \lambda\in \bC^*.\end{equation}
The polynomials $B_k(x)=B_{1,k}(x\b 1)$ coincide with the classical Bernoulli polynomials.
	
It will be useful to have explicit expressions for a few of these polynomials to hand:
\[B_{1,0}(x\b\omega_1)=\frac{1}{\omega_1}, \qquad B_{1,1}(x\b \omega_1) = \frac{x}{\omega_1}-\frac{1}{2},\qquad B_{1,2}(x\b a_1)=\frac{x^2}{a_1}-x+\frac{a_1}{6}.\]
\[B_{2,0}(x \b \omega_1,\omega_2)= \frac{1}{\omega_1\omega_2},\qquad B_{2,1}(x \b \omega_1,\omega_2)=\frac{x}{\omega_1\omega_2} -\frac{\omega_1+\omega_2}{2\omega_1\omega_2},\]
\[B_{2,2}(x\b\omega_1,\omega_2)=\frac{x^2}{\omega_1\omega_2} -\Big(\frac{1}{\omega_1} + \frac{1}{\omega_2}\Big) x+\frac{1 }{6}\Big(\frac{\omega_2}{\omega_1} + \frac{\omega_1}{\omega_2}\Big)+\frac{1}{2}.\]
These can be  obtained by multiplying out the classical Bernoulli polynomial expansions for the $N$ individual factors appearing on the left of \eqref{Bpoly}.

\subsection{Multiple gamma  functions}\label{sec_barnes}
We recall here the definition of the Barnes multiple gamma functions.  Our basic references  for this material are \cite{FR,JM, Ruj2}. Most of the results we need can also be found in Barnes' original papers \cite{Barnes2,Barnes3}, although it is important to note that these older sources use a different normalization: see  \cite[Equation (3.19)]{Ruj2}.

We again fix a positive integer $N>0$ and a vector of non-zero complex numbers \begin{equation*}
\underline{a}=(a_1,\dots, a_N)\in (\bC^*)^N.\end{equation*} Assume for now that $\Re(a_i)>0$ for all $i$.  Let us also fix an element $x\in \bC$. The Barnes multiple zeta function is defined by the sum
\begin{equation*}
	\zeta_N(s,x \b \underline a)=\big.\sum_{\underline n \in (\bZ_{\geq 0})^N} (x+\underline n \cdot \underline a)^{-s},		
	\end{equation*}
which is absolutely convergent for $\Re(s)>N$. It can be analytically continued, \cite[Section 3]{Ruj2}, to a single-valued meromorphic function of $s\in \bC$, with poles only at the points $s=1,2, \cdots, N$.

Assuming again that $\Re(\omega_i)>0$, the Barnes multiple gamma function $\Gamma_N$ is defined by the formula
	\begin{equation*}
	\Gamma_N(x\b\underline a) := \exp \frac{\del}{\del s} \zeta_N(s, x \b \underline a)|_{s=0}.
	\end{equation*}
	This is a meromorphic function of $x\in \bC$, without zeroes, and whose poles,  \cite[Section 3]{Ruj2}, lie at the points of the form \[x=-\sum_{i=1}^N m_i \omega_i, \qquad m_i\in \bZ_{\geq 0}.\] 
	
	The functions $\zeta_N(s,x\b \underline a)$ and $\Gamma_N(x\b\underline a)$ are generalizations of the Hurwitz zeta function $\zeta_H(s,x)$ and the gamma function $\Gamma(x)$ respectively. Indeed, \cite[Equations (3.23) and (3.27)]{Ruj2} give \begin{equation}\label{bl}\zeta_1(s,x\b a)= a^{-s}\cdot \zeta_H\left(\frac{x}{a}\right),\qquad \Gamma_1(x\b a)= \frac{1}{\sqrt{2\pi}}\cdot \Gamma\left(\frac{x}{a}\right) \cdot a^{\frac{x}{a}-\frac{1}{2}},\end{equation}
	where we take the principal branch of $\log(a)$ on the half-plane $\Re(a)>0$.

The main property of the zeta function we shall use is the difference equation
	\begin{equation*}
	\zeta_{N}(s,x\b a_1,\cdots a_N) - \zeta_{N}(s,x+ a_i\b a_1,\cdots,a_N) = \zeta_{N-1}(s,x\b a_1,\cdots, a_{i-1},a_{i+1},\cdots,a_N),
	\end{equation*}
which is immediate from the definition. This relation  induces an analogous relation for $\log\Gamma_N$.

The multiple gamma functions have the homogeneity property 
\begin{equation}\label{homogeneity}
	\Gamma_N(\lambda \cdot x \b \lambda \cdot \underline a)=\exp\left(\frac{1}{N!}\cdot (-1)^{N-1}\cdot B_{N,N}(x \b \underline a)\cdot \log(\lambda)\right) \cdot \Gamma_N(x \b \underline a),
	\end{equation}
	valid for $\lambda\in \bC^*\setminus \bR_{<0}$ such that $\Re(\lambda\cdot a_i)>0$ for all $i$. We take the principal branch of $\log(\lambda)$. This follows immediately from the definition once one knows  \cite[Appendix A]{JM} that
\begin{equation*}
\zeta_N(0,x \b \underline a)=\frac{(-1)^N}{N!} \cdot  B_{N,N}(x \b \underline a).\end{equation*}
The relation \eqref{homogeneity}  allows us to analytically continue the function $\Gamma_N(x\b \underline a)$ to the domain
\[a_i\in \bC^*\setminus \bR_{<0}, \qquad \Re(a_i/a_j)>0, \qquad 1\leq i,j\leq N.\]
In words, we allow the parameters $a_i$ to vary freely in the domain $\bC^*\setminus \bR_{<0}$ providing that they all lie in a single open half-plane.

\subsection{Modified gamma function}\label{sec_Lambdafun}
It will be useful to consider certain modifications of the Barnes gamma functions designed to kill the sporadic terms in their asymptotics as $x\to \infty$. We first consider the case $N=1$. 
Take $\omega\in \bC^*$ with $a\in \bC^*\setminus\bR_{<0}$ and consider the  function
\[\Lambda(w,\eta\b\omega)=\Gamma_1(w+\eta\b \omega)\cdot \exp(-B_{1,1}(w+\eta\b \omega)\log(w))\cdot \exp\Big(\frac{w}{\omega}\Big),\]
where $w\in \bC^*\setminus\bR_{<0}$ and $\eta\in \bC$, and we take the principal branch of $\log(w)$. Using \eqref{bl} we can rewrite this as
\begin{equation}
\label{dall}\Lambda(w,\eta\b\omega)=(2\pi)^{-\frac{1}{2}}\cdot \Gamma\Big(\frac{w+\eta}{\omega}\Big)\cdot \exp\Big(\frac{w}{\omega}\Big)\cdot \Big(\frac{w}{\omega}\Big)^{\frac{1}{2}-\frac{w+\eta}{\omega}},\end{equation}
although one should be a little careful here, since with our chosen analytic continuations, the expression $\log(w/a)=\log(w)-\log(a)$ is specified by the principal branches of the functions $\log(w)$ and $\log(a)$  on the domain $\bC^*\setminus \bR_{<0}$. 

\begin{prop}\label{lemGamma} The function $\Lambda(w,\eta\b \omega)$ has the following properties:
\begin{itemize}
\item[(a)] It is a single-valued, meromorphic  function of the variables $w,\omega\in \bC^*\setminus\bR_{<0}$ and  $\eta\in \bC$. It has no zeroes, and poles only at the points
 \[w+\eta=n \omega, \qquad n\in \bZ_{\leq 0}.\]

\item[(b)]  It has a homogeneity property
\begin{equation}
\label{homogg}\Lambda(\lambda w,\lambda \eta\b \lambda \omega)=\Lambda(w,\eta\b \omega),\end{equation}
for $\lambda\in \bC^*$ such that $\lambda w,\lambda \omega\in \bC^*\setminus\bR_{<0}$.\smallskip

\item[(c)] On the half-plane $\pm \Im(w/\omega)>0$ it satisfies the reflection property 
 \begin{equation}\Lambda(w,\eta\b\omega)\cdot \Lambda(-w,\omega-\eta\b\omega) = 
\left(1 - e^{\pm \frac{2\pi i (w+\eta)}{\omega}}\right)^{-1}.\end{equation}

\item[(d)] For fixed $\omega\in \bC^*\setminus\bR_{<0}$ and $\eta\in \bC$ there is a constant $k>0$  such that for $0<|w|\ll 1$
 \[|w|^{k}<|\Lambda(w,\eta\b\omega)|<|w|^{-k}.\]

\item[(e)] When $a\in \bR_{>0}$  and $\eta\in \bC$ there is an asymptotic expansion
	\[
	\log \Lambda(w,\eta\b\omega)\sim  \sum_{k\geq 1} \frac{(-1)^{k+1}B_{1,k+1}(\eta\b\omega)}{k(k+1)} \cdot w^{-k}.
	\]
valid as $|w|\to \infty$ in any closed subsector of $\bC^*\setminus\bR_{<0}$.
	\end{itemize}
\end{prop}
\begin{proof}
Properties (a), (b) and (d) are clear from expression \eqref{dall} and well-known properties of the gamma function. For (c) we can use the homogeneity property to reduce to the case $a=1$, when the given relation is a simple consequence of the Euler reflection formula for the gamma function: see   Lemma 3.1 in \cite{RHuncoupled} for more details. For (e) we can reduce to the case $a=1$ using the homogeneity properties \eqref{homoBnk} and \eqref{homogg}, when the claim is a form of the Stirling expansion. \end{proof}

\subsection{Modified double gamma function}\label{sec_F}
Take parameters $\omega_1,\omega_2\in \bC^*\setminus \bR_{<0}$ and assume  that $\Re(\omega_2/\omega_1)>0$.  In this section we consider the function
\begin{equation}\begin{split}
\label{christmas}
F(w,\eta \b \omega_1,\omega_2):= \Gamma_2(w+\eta \b \omega_1,\omega_2)\cdot \exp\left(\half B_{2,2}(w+\eta \b \omega_1,\omega_2) \log w\right)\cdot\qquad \\
\cdot \exp\bigg({-\frac{3w^2}{4\omega_1\omega_2}-\frac{\eta w}{\omega_1\omega_2}+\frac{w(\omega_1+\omega_2)}{2\omega_1\omega_2}}\bigg),
\end{split}\end{equation}
with $w\in \bC^*\setminus \bR_{<0}$ and $\eta\in \bC$. As before, the middle factor is fixed by choosing the principal branch of $\log(w)$. This function is a modification of the Barnes double gamma function obtained by killing the sporadic terms in its asymptotics as $w\to \infty$.

\begin{prop}
\label{mod}
The function $F(w,\eta \b \omega_1,\omega_2)$ has the following properties:
\begin{itemize}
\item[(a)] It is a single-valued, meromorphic function of   $w,\omega_1,\omega_2\in \bC^*\setminus \bR_{<0}$ and $\eta\in \bC$ providing that $\Re(\omega_2/\omega_1)>0$. It has no zeroes, and poles  only at the points
\[w+\eta=n_1\omega_1+n_2\omega_2, \qquad (n_1,n_2)\in \bZ_{\leq 0}^2.\]

\item[(b)] It satisfies the symmetry relation
\[F(w,\eta \b \omega_1,\omega_2)=F(w,\eta \b \omega_2,\omega_1),\]and  the homogeneity relation\[F(\lambda w,{\lambda \eta} \b \lambda \omega_1,\lambda \omega_2)=F(w,\eta \b \omega_1,\omega_2),\]
valid for $\lambda\in \bC^*$ such that $\lambda w, \lambda\omega_i \in \bC^*\setminus\bR_{<0}$.
\smallskip

\item[(c)] It satisfies a  difference relation
\[\frac{F(w,{\eta+\omega_2}\b \omega_1,\omega_2)}{F(w,\eta \b \omega_1,\omega_2)}=\Lambda(w,\eta\b \omega_1)^{-1}.\]

\item[(d)] Consider fixed $\omega_1,\omega_2\in \bC^*$ and $\eta\in \bC$ and assume that $\Re(\omega_i)>0$. Then there is an asymptotic expansion 
	\[\log F(w,\eta \b \omega_1,\omega_2) \sim
\sum_{k\geq 1}\frac{(-1)^k\cdot B_{2,k+2}(\eta\b \omega_1,\omega_2)}{k(k+1)(k+2)} \cdot w^{-k},
	\]
valid as  $|w|\to \infty$ in  any closed subsector of the half-plane $\Re(w)>0$.
\end{itemize}
\end{prop}

\begin{proof}
The single-valuedness of $F$ is a consequence of the definition for $\Re(a_i)>0$, and the way we have analytically continued to the domain $a_i\in \bC^*\setminus\bR_{<0}$ and $\Re(a_2/a_1)>0$. The claim about the zeroes and poles follows  from the corresponding properties of the double gamma function which can be found in \cite[Appendix A]{JM}.
The symmetry in $\omega_1,\omega_2$ is immediate from the definition of the double gamma function. The homogeneity of $F(w,\eta\b\omega_1,\omega_2)$ is a consequence of \eqref{homogeneity}.

For part (c) start with the reflection identity for the double gamma function
\[\frac{\Gamma_2(x+\omega_2 \b \omega_1,\omega_2)}{\Gamma_2(x \b \omega_1,\omega_2)}=\Gamma_1(x\b a_1)^{-1},\]
which can be  found in \cite[Appendix A]{JM} under the assumption $\Re (\omega_i)>0$. The result follows by analytic continuation, and  the identity
\begin{equation*}
	\label{der}B_{2,2}(x+\omega_2 \b \omega_1,\omega_2)-B_{2,2}(x \b \omega_1,\omega_2)=2B_{1,1}(x\b\omega_1).
\end{equation*}
The proof of part (d) can be found in the Appendix as a consequence of a more general statement about multiple gamma functions $\Gamma_N(x+\delta\b\underline{a})$.
\end{proof}


\section{The quantum Riemann-Hilbert problem for doubled A$_1$ }\label{sec_statement}

In this section we describe the refined BPS structure associated to the double of the  A$_1$ quiver, and show how it defines a rigorous quantum Riemann-Hilbert problem taking values in the group of automorphisms of an extension of the quantum torus algebra. We then give a solution to this problem using the special functions  of Section \ref{sec_special_funct}.

\subsection{The doubled A$_1$ example}
\label{se}

Let $(\Gamma,Z,\Omega)$ be a refined BPS structure as defined in the introduction. The corresponding quantum torus algebra is the non-commutative ring \begin{equation}
\label{alg}\bC_{q}[\bT]=\bigoplus_{\gamma\in \Gamma} \bC\big[\bL^{\pm \half}\big]\cdot x_\gamma,\qquad x_{\gamma_1}*x_{\gamma_2}= \bL^{\half \bra\gamma_1,\gamma_2\ket} \cdot x_{\gamma_1+\gamma_2}.\end{equation}
It is an algebra over the ring of Laurent polynomials  $\bC\big[\bL^{\pm \half}\big]$.
The specialisations at $\bL^{\half}=\pm 1$ are the commutative algebras
\begin{equation*}
\bC[\bT_\pm]=\bigoplus_{\gamma\in\Gamma} \bC\cdot x_\gamma, \qquad x_{\gamma_1}* x_{\gamma_2}=(\pm 1)^{\<\gamma_1,\gamma_2\>}\cdot x_{\gamma_1+\gamma_2},\end{equation*}
which are the rings of algebraic functions on the varieties
\begin{equation}
\label{oldy}\bT_\pm =\Big\{\xi\colon \Gamma\to \bC^* \b \xi(\gamma_1+\gamma_2)=(\pm 1)^{\langle\gamma_1,\gamma_2\rangle}\cdot \xi(\gamma_1)\cdot \xi(\gamma_2)\Big\}\isom (\bC^*)^2.\end{equation}
These are the algebraic torus $\bT_+$ and the twisted torus $\bT_-$ considered in \cite[Section 2]{RHDT}.

In this section we shall consider the  refined BPS structures $(\Gamma,Z,\Omega)$ defined in  Example \ref{mineral}. Recall that they consist of the following data:
\begin{itemize}
\item[(a)] the lattice $\Gamma\isom \bZ^{\oplus 2}$ and skew-symmetric form $\<-,-\>$ are
\[\Gamma=\bZ\cdot \alpha\oplus \bZ\cdot \alpha^\dual, \qquad \<\alpha^\dual,\alpha\>=1.\]

\item[(b)] the group homomorphism $Z\colon \Gamma\to \bC$ is determined by \[z=Z(\alpha)\in \bC^*, \qquad  Z(\alpha^\vee)=0.\]

\item[(c)] the only non-zero refined BPS invariants are $\Omega(\pm\alpha)=1$. 
\end{itemize}

\begin{remark}\label{physics}From a mathematical point-of-view, this data arises from the refined BPS structure defined by the Donaldson-Thomas theory of the A$_1$ quiver by a formal doubling procedure \cite[{Section 2.6}]{KS}. In physical terms it corresponds \cite[Section 4]{GMN1} to the $U(1)$ gauge theory whose charge lattice $\Gamma$ is spanned by \lq\lq electric\rq\rq\ and \lq\lq magnetic\rq\rq\ generators $\gamma_e,\gamma_m$ satisfying $\<\gamma_m,\gamma_e\>=1$, and whose only nonzero BPS invariants are $\Omega(\pm \gamma_e)=1$.\end{remark}

\subsection{Quadratic refinement}
\label{quad}

As in the introduction it will be convenient to define some alternative generators for the quantum torus by introducing some signs.  This is fiddly but really just a matter of convention, and can safely be ignored at a first reading.

A quadratic refinement of the form $\<-,-\>$ is a point of the finite subset
\[\Big\{\sigma\colon \Gamma\to \{\pm 1\} \b \sigma(\gamma_1+\gamma_2)=(-1)^{\langle\gamma_1,\gamma_2\rangle}\cdot \sigma(\gamma_1)\cdot \sigma(\gamma_2)\Big\}\subset \bT_-.\]
Such a point $\sigma\in \bT_-$ defines an involution of the quantum torus algebra
\begin{equation}
\label{nhm}D\colon \bC_q[\bT]\to \bC_q[\bT], \qquad \bL^{\half}\mapsto -\bL^{\half}, \quad x_\gamma\mapsto \sigma(\gamma)\cdot x_\gamma.\end{equation}
Note that this automorphism exchanges the two commutative limits $\bL^{\half}\to \pm 1$ considered above, and therefore induces an isomorphism $\bT_+\isom \bT_-$. 

For the doubled A$_1$ refined BPS structure we are considering, an  example of such a quadratic refinement can be defined by setting
\begin{equation}
\label{standard}\sigma(m\alpha+n\alpha^\vee)=(-1)^{m(n+1)}.\end{equation}
Although  this  definition looks rather arbitrary at first sight, this quadratic refinement should in fact be viewed as being canonical (see the discussion in \cite[Section 7.7]{GMN2}): it is uniquely defined by the property that $\Omega(\gamma)\neq 0$ implies  $ \sigma(\gamma)=-1$.

We would now like to compose the quantum Riemann-Hilbert by the involution $D$ so as to obtain nicer formulae for its solution.  To express this in a  down-to-earth manner, we introduce new variables\footnote{The choice of notation  here is a little unfortunate since often in the literature one finds  $q^{\half}=\bL^{\half}$ instead. } 
 \[y_\gamma:=\sigma(\gamma)\cdot x_\gamma, \qquad q^{\half}:=-\bL^{\half}.\]
 Then the quantum torus is
 \[\bC_{q}[\bT]=\bigoplus_{\gamma\in \Gamma} \bC\big[q^{\pm \half}\big]\cdot y_\gamma,\qquad y_{\gamma_1}*y_{\gamma_2}= q^{\half \bra\gamma_1,\gamma_2\ket} \cdot y_{\gamma_1+\gamma_2}.\]
 This change of variables has no effect on the form of the heuristic quantum Riemann-Hilbert problem described in the Introduction, and  at first sight looks completely trivial. However, it becomes non-trivial when we pass to the extended quantum torus.

\subsection{Extended quantum torus algebra}

Let us recall from the introduction the definition of the extended quantum torus algebra\begin{equation}
\label{grading}\widehat{\bC_q[\bT]}=\bigoplus_{n\in \bZ} \Mer(\cH\times \bC) \cdot y_{n\alpha^\dual},\end{equation}
where $\Mer(\cH\times \bC)$ denotes the field of meromorphic functions $f(\tau,\theta)$ on the product of the upper half-plane $\cH$ with the complex plane $\bC$, and 
the product is
\begin{equation}
\label{mult_hat}\Big(f_1(\tau,\theta)\cdot y_{n_1\alpha^\dual}\Big) *\Big( f_2(\tau,\theta)\cdot y_{n_2\alpha^\dual}\Big)={ f_1(\tau,\theta)\cdot f_2(\tau,\theta+n_1\tau)}\cdot y_{(n_1+n_2)\alpha^\dual}.\end{equation}
We also consider the commutative subalgebra
\begin{equation}\label{comm2}\widehat{\bC_q[\bT]}_0=\Mer(\cH\times \bC)\cdot 1 \subset \widehat{\bC_q[\bT]}.\end{equation}

\begin{lemma}
\label{emb}
There is an injective ring homomorphism
$I\colon \bC_q[\bT]\into \widehat{\bC_q[\bT]}$ defined by \[q^{\frac{k}{2}} \cdot y_{m\alpha+n\alpha^\dual}\mapsto \exp\big( \pi i  (k+mn)\tau+2\pi i m\theta\big)\cdot y_{n\alpha^\dual}.\]
\end{lemma}

\begin{proof}
The relations in \eqref{alg} are easily checked. The fact that the resulting ring homomorphism is injective follows from the fact that for any distinct complex numbers $a_1,\cdots, a_n\in \bC$, the exponential functions $f_i(t)=\exp(a_i t)$ are linearly independent over $\bC$. \end{proof}

We often identify elements of $\bC_q[\bT]$ with their images under the embedding $I$. Note, in particular, that
\begin{equation}
\label{ident}I(q^{\half})=\exp(\pi i \tau)\cdot 1, \qquad I(y_\alpha)=\exp(2\pi i \theta)\cdot 1, \qquad I(y_\alpha^{\vee})=y_{\alpha^\vee}.\end{equation}
The group homomorphism $Z\colon \Gamma\to \bC$ defines a family of automorphisms
\[\epsilon_Z(t)\in \Aut \bC_q[\bT], \qquad \epsilon_Z(t)(y_{\gamma})=e^{Z(\gamma)/t}\cdot y_\gamma, \qquad \gamma\in \Gamma,\ t\in\bC^*,\]
which lift the  rotations of the tori $\bT_\pm $ obtained by exponentiating the flows of the invariant vector fields corresponding to $Z/t$. These automorphisms extend to a family of automorphisms of $\widehat{\bC_q[\bT]}$ defined by
\[\epsilon_Z(t)\Big(f(\tau,\theta)\cdot y_{n\alpha^\vee}\Big)=f\Big(\tau,\theta+\frac{z}{2\pi i t}\Big)\cdot y_{n\alpha^\vee},\]
where we used the assumption  that $Z(\alpha^\vee)=0$.

\subsection{Quantum dilogarithm}
The quantum dilogarithm  function is defined by the infinite product 
\[\bE_q(x)=\prod_{k\geq 0} (1-q^k x).\]
Under the assumption $|q|<1$ the product converges absolutely and defines a nowhere-vanishing analytic function of $x\in \bC$. Assuming $|x|<1$ we can expand \cite[Section 1.3]{FG} as
\[\bE_q(x)^{-1}=\sum_{n\geq 0} \frac{ x^n }{(1-q)\cdots (1-q^n)} =\exp_q \Big(\frac{x}{1-q}\Big),\]
where the quantum exponential is
\[\exp_q(x)=\sum_{n\geq 0} \frac{x^n}{[n]_q !}, \qquad [n]_q!=[n]_q\cdot [n-1]_q \cdots [1]_q, \qquad [k]_q=\frac{q^k-1}{q-1}.\]

\begin{remark}
\label{convent}
There are different conventions for this function in the literature. For example,  Kontsevich and Soibelman \cite[Section 6.4]{KS} define \[\bE^{KS}_{q^{\frac{1}{2}}}(x)=\bE_q(-q^{\frac{1}{2}}x)^{-1}=1+\sum_{n\geq 1}\frac{q^{\frac{n^2}{2}}\cdot x^n}{(q^n-1)\cdots (q^n-q^{n-1})}.\]
This is also the convention used by Fock and Goncharov \cite[Section 1.3]{FG}, although they refer to this function as the q-exponential, reserving the term quantum dilogarithm for a different function, often called the Fadeev quantum dilogarithm, which is essentially the double sine function. 
\end{remark}

We consider the elements
\begin{equation*}
	\DT_q(\ell_\pm)=\bE_{e^{2\pi i \tau}}\big(\!  -\! e^{\pi i \tau\pm 2\pi i \theta}\big)^{-1} \cdot 1\in \widehat{\bC_q[\bT]}_0,\end{equation*}
which are clearly invertible, and the corresponding automorphisms
\[\bS_q(\ell_\pm)=\Ad_{\DT_q(\ell_\pm)}\in \Aut  \widehat{\bC_q[\bT]}.\]
We can compute these automorphisms explicitly as follows.

\begin{lemma}
\label{oneone}
The automorphisms $\bS_q(\ell_\pm)$
act trivially on the subalgebra \eqref{comm2} and satisfy 
\[\bS_q(\ell_\pm)\colon y_{\alpha^\vee}\mapsto  (1+q^{\pm \frac{1}{2}} y_{\pm \alpha})^{\mp 1}*y_{\alpha^\vee}.\]\end{lemma}

\begin{proof}
The automorphisms $\bS_q(\ell_\pm)$ act trivially on the subalgebra $\widehat{\bC_q[\bT]}_0$ because this subalgebra is commutative.
The definition of the quantum exponential shows that it satisfies the difference relation
\[\bE_q(x)\cdot \bE_q(qx)^{-1}=1-x.\]
Using the definitions we therefore have
\[\bE_q(-q^{\frac{1}{2}}y_{\pm \alpha})^{-1}* y_{\alpha^\dual}*\bE_q(-q^{\frac{1}{2}} y_{\pm \alpha})=\bE_{e^{2\pi i \tau}}(-e^{\pm 2\pi i \theta+\pi i \tau})^{-1}\cdot 1  * y_{\alpha^\dual} * \bE_{e^{2\pi i \tau}} (-e^{\pm 2\pi i \theta+\pi i \tau})\cdot 1\]\[=\bE_{e^{2\pi i \tau}} (-e^{\pm 2\pi i \theta+\pi i \tau})^{-1}\cdot \bE_{e^{2\pi i \tau}} (-e^{\pm 2\pi i \theta+(1\pm 2)\pi i \tau})\cdot y_{\alpha^\vee}=(1+e^{\pm 2\pi i \theta\pm\pi i \tau})^{\mp 1} \cdot y_{\alpha^\vee},\]
which gives the stated result under the identifications \eqref{ident}. \end{proof}

\subsection{Maps into the extended quantum torus}
Before stating the quantum Riemann-Hilbert problem we need to make a few definitions concerning holomorphic maps into the extended quantum torus algebra, and the limiting behaviour of such maps. 

\begin{definitions}
Let $D\subset \bC$ be a domain.

\begin{itemize}
\item[(a)]By a meromorphic map
$f\colon  D\to \widehat{\bC_q[\bT]} $
 we mean a finite sum of the form
 \[f(t)=\sum_n f_n(\tau,\theta,t)\cdot y_{n\alpha^\vee},\]
such that each function $f_n(\tau,\theta,t)$ is meromorphic on $\cH\times\bC\times D$.\smallskip

\item[(b)]Given a meromorphic map $f(t)$ as in (a), and a point  $t_0\in \bar{D}$, we say that \[f(t)\to g=\sum_n g_n(\tau,\theta)\cdot y_{n\alpha^\vee}\in\widehat{\bC_q[\bT]}\]
as $t\to t_0$, if for each $n\in \bZ$, and each $(\tau,\theta)\in \cH\times\bC$,  one has \[f_n(\tau,\theta,t)\to g_n(\tau,\theta)\text{ as }t\to t_0.\] 

\item[(c)] Suppose  the domain $D$ is unbounded.  We say that a meromorphic map $f(t)$ as in (a) has bounded growth at infinity if for all $(\tau,\theta)\in \cH\times\bC$, and all $n\in \bZ$, there is a $k>0$ with
\[|t|^{-k} < |f_n(\tau,\theta,t)|<|t|^k\text{ as } |t|\to \infty.\]

\item[(d)]
  We say that a map
 $\phi\colon D\to \Aut \widehat{\bC_q[\bT]}$
 is meromorphic if for any element $a\in  \widehat{\bC_q[\bT]}$ the map
\[\operatorname{eval}_a(\phi)\colon D\to \widehat{\bC_q[\bT]},\qquad \operatorname{eval}_a(\phi)(t)=\phi(t)(a)\]
obtained by applying  $\phi(t)\in  \Aut \widehat{\bC_q[\bT]}$ to the element $a\in  \widehat{\bC_q[\bT]}$ is meromorphic in the sense of (a).\smallskip

\item[(e)] We similarly extend definitions (b) and (c) to the case of automorphisms by evaluating on arbitrary elements. Thus, for example, given a meromorphic map $\phi(t)$ as in (d), and a point $t_0\in \bar{D}$, we  say that \[\phi(t)\to \psi\in \Aut \widehat{\bC_q[\bT]}\text{ as }t\to t_0,\] if for every element $a\in  \widehat{\bC_q[\bT]}$, one has $\operatorname{eval}_a(\phi)(t)\to \operatorname{eval}_a(\psi)$ in the sense of (b). Similarly for bounded growth at infinity.
\end{itemize}
\end{definitions}

\subsection{Rigorous quantum Riemann-Hilbert problem}\label{sec_RH}
We can now state a rigorous version of the quantum Riemann-Hilbert problem for the doubled A$_1$ refined BPS structure. As in \cite[Section 4]{RHDT} it is best to formulate the problem in terms of maps on half-planes in $\bC$ centered on non-active rays (see \cite[Remark 4.6 (ii)]{RHDT}). Glueing these together exactly as in \cite[Section 5.1]{RHDT} we arrive at the following formulation.

\begin{problem}\label{probl_trasl}We look for meromorphic maps
\[\Phi_\pm \colon \bC^*\setminus  i\ell_{\pm}   \to  \Aut \widehat{\bC_q[\bT]},\]
satisfying the three properties
\begin{itemize}
\item[(qRH1)] There are relations 
	\begin{equation*}
	\Phi_+(t)=\begin{cases}  \Phi_-(t)\circ \bS_q(\ell_+) 
		&\mbox{ if }\Re(t/z)>0,\\ \Phi_-(t)\circ \bS_q(\ell_-)^{-1}   &\mbox{ if }\Re(t/z)<0.\end{cases}
	\end{equation*}

\item[(qRH2)] As $t\to 0$ in $\bC^*\setminus i\ell_{\pm}$ we have $\Phi_{\pm}(t) \circ \epsilon_Z(t)\to \id\in\Aut \widehat{\bC_q[\bT]}$.
\smallskip

\item[(qRH3)] The map $\Phi_{\pm}(t)$ has bounded growth at infinity.
\end{itemize}
\end{problem}

Written in terms of the equivalent data
\[\Psi_\pm \colon \bC^*\setminus  i\ell_{\pm}   \to  \Aut \widehat{\bC_q[\bT]},\qquad \Psi_\pm(t)= \Phi_\pm(t)\circ \epsilon_Z(t),\]
the condition (qRH1) becomes
\begin{equation*}
\Psi_+(t)=\begin{cases}  \Psi_-(t)\circ \tilde{\bS}_q( \ell_+)  &\mbox{ if }\Re(t/z)>0,\\  \Psi_-(t)  \circ \tilde{\bS}_q( \ell_-)^{-1} &\mbox{ if }\Re(t/z)<0,\end{cases}
\end{equation*}
where we define automorphisms
\[\tilde{\bS}_q(\ell_\pm)=\epsilon_Z(-t)\circ \bS_q(\ell_\pm)\circ \epsilon_Z(t)=\Ad_{\epsilon_Z(-t)(\bE_q(-q^\frac{1}{2}y_{\pm \alpha})^{-1})}.\]
These automorphisms again act trivially on the subalgebra \eqref{comm2} and  satisfy
\begin{equation}
	\label{lab}\tilde{\bS}_q(\ell_\pm)^{\pm 1}\colon y_{\alpha^\vee}\mapsto (1+q^{\pm \frac{1}{2}} \cdot e^{\mp \frac{z}{t}}\cdot y_{\pm \alpha})^{-1}*y_{\alpha^\vee} .\end{equation}
The conditions (qRH2) and (qRH3) are unchanged.

\subsection{The solution}
\label{sec_soln}

The form of the discontinuites \eqref{lab} makes the problem identical to the classical commutative Riemann-Hilbert problem for $A_1$ studied in \cite{RHuncoupled} after a shift $\theta\mapsto \theta+\half \tau$.  The choice 
of the quadratic refinement produces an additional shift by $\theta\mapsto \theta+\half$. We can therefore give the solution as follows. 

\begin{theorem}\label{prop_pa}
Problem \ref{probl_trasl} has solutions the automorphisms $\Psi_{\pm}(t)$ which act trivially on the subalgebra \eqref{comm2} and satisfy
	\begin{equation}
		\label{soln}		\Psi_\pm (t)\left(y_{\alpha^{\vee}}\right)= \Lambda\left(\pm \frac{z}{2\pi i t}, \frac{1}{2}\mp\Big(\theta+\frac{\tau}{2}\Big) \wb 1\right)^{\pm 1}\cdot y_{\alpha^{\vee}} .\end{equation}\end{theorem}
	
\begin{proof}
All automorphisms of $\widehat{\bC_q[\bT]}$ being considered act trivially on the subalgebra \eqref{comm2} so it will be enough to consider their action on the generator $y_{\alpha^\vee}$. Using formula \eqref{lab} the jumping condition (qRH1)  comes down to the statement that 
\[\Lambda\left( \frac{z}{2\pi i t}, \frac{1}{2}- \Big(\theta+\frac{\tau}{2}\Big) \wb 1\right)=\Lambda\left(-\frac{z}{2\pi i t},  \frac{1}{2}+\Big(\theta+\frac{\tau}{2}\Big) \wb 1\right)^{-1}\cdot \left(1+e^{\pm \pi i\tau}\cdot e^{\mp \frac{z}{t}}\cdot e^{\pm 2\pi i \theta}\right )^{-1},\]
when $\pm \Re(t/z)>0$, which follows from Lemma \ref{lemGamma}(c). Since all elements of the algebra $\widehat{\bC_q[\bT]}$ are polynomials in the element $y_{\alpha^\vee}$ over the subalgebra \eqref{comm2}, to check (qRH2) it is enough to check that
\[\Lambda\left(\pm \frac{z}{2\pi i t}, \frac{1}{2}\mp\Big(\theta+\frac{\tau}{2}\Big) \wb 1\right)\to 1\]
as $t\to 0$ in the  domain $\pm \Re(t/z)>0$, which  follows from Lemma \ref{lemGamma}(f). Similarly the bounded growth condition (qRH3) follows from Lemma \ref{lemGamma}(e).
\end{proof}

We then have  for $n>0$
\[
	\Psi_{\pm}(t)\left(y_{n\alpha^\vee}\right) = \Psi_{\pm}(t)\left(y_{\alpha^\vee}\right)*\dots *\Psi_{\pm}(t)\left(y_{\alpha^\vee}\right) =
	\prod_{j=0}^{n-1} \Lambda\left(\pm \frac{z}{2\pi i t},\frac{1}{2}\mp \Big(\theta+\Big(j+\frac{1}{2}\Big)\Big) \tau\wb 1\right)^{\pm 1}\cdot y_{n\alpha^\vee}
\]
	
\[
	\Psi_\pm (t)\left(y_{-n\alpha^\vee}\right) =  \Psi_\pm (t)\left(y_{n\alpha^\vee}\right)^{-1} 
	= \prod_{j=-n}^{-1} \Lambda\left(\pm \frac{z}{2\pi i t},\frac{1}{2}\mp \Big(\theta+ \Big(j+\frac{1}{2}\Big)\Big) \tau\wb 1\right)^{\mp 1}\cdot y_{-n\alpha^\vee}
	\]
	
\begin{remark}
\label{uniquness}
We  can only make very weak uniqueness statements for the above solution. Suppose we impose the extra condition that the solution $\Psi_\pm(t)$ should preserve the grading
\[\widehat{\bC_q[\bT]}=\bigoplus_{n\in \bZ} \widehat{\bC_q[\bT]}_0\cdot y_{n\alpha^\vee},\]
and act trivially on the zeroth graded piece. Any such solution is determined by
\[\Psi_\pm(t)\colon y_{\alpha^\vee}\mapsto f_\pm(\theta,\tau,z,t)\cdot y_{\alpha^\vee}\]
for some meromorphic functions $f_\pm(\theta,\tau,z,t)$. The transformations \eqref{lab} have poles or zeroes at the point 
\begin{equation}
\label{t}t= \frac{z}{2\pi i \left(n+\theta+\half(1+\tau)\right)} ,\end{equation}
for all integers $n\in \bZ$, so it is inevitable that the functions $f_\pm$ also have poles or zeroes at these points. If we impose the condition that  $f_\pm(\theta,\tau,z,t)$, considered as functions of $t\in \bC^*\setminus\ell_\pm$,  have finitely many poles and zeroes, and that these are simple and occur only at the points \eqref{t}, then 
similar arguments to \cite[{Section\,2}]{RHuncoupled} show that  they must be given by the formula \eqref{soln} up to shifting the variable $\theta$ by an integer.
\end{remark}

\subsection{The adjoint form}
It is interesting to write the  solution of Theorem \ref{prop_pa} in adjoint form as follows.

\begin{theorem}\label{thm_af} For each $z\in \bC^*$ the solution $\Psi_{\pm}(t)$ of Theorem   \ref{prop_pa} can  be expressed as 
	\begin{equation*}
	\Psi_{\pm}(t) = \Ad_{\psi_{\pm}(t)},\qquad\psi_{\pm}: \bC^*\setminus i\ell_{\pm} \to \widehat{\bC_q[\bT]},
	\end{equation*}
where  we define maps\begin{equation}
	\label{easter}\psi_{\pm}(t)=F\Big(\pm \frac{z}{2\pi i t},\frac{1+\tau}{2}\mp \theta\wb1, \tau\Big)^{-1}.\end{equation}
\end{theorem}
\begin{proof}
We recall the multiplication rule \eqref{mult_hat}, which shows that for all $g\in\Mer(\cH\times \bC)$
	\begin{equation*}
	g(\tau ,\theta) * y_{\alpha^{\vee}} * g(\tau,\theta)^{-1} = g(\tau,\theta)g(\tau,\theta+\tau)^{-1}\cdot y_{\alpha^{\vee}}.
	\end{equation*}
The claim then follows from Proposition \ref{mod}\,(c). 
\end{proof}

\begin{remark}The above adjoint form is far from unique. 
To reduce this indeterminacy one could try to lift the quantum Riemann-Hilbert problem to a problem involving maps
\[\psi_\pm \colon \bC^*\setminus i\ell_{\pm}\to \widehat{\bC_q[\bT]},\]
with the jumping conditions
\[\psi_+(t)=\begin{cases} \psi_-(t) * \bE_{e^{2\pi i \tau}}\left(-\exp\left(\pi i \tau+2\pi i \theta-\frac{z}{t}\right)\right)^{-1} &\mbox{ if }\Re(t/z)>0,\\ \psi_-(t)*\bE_{e^{2\pi i \tau}}\left(-\exp\left(\pi i \tau-2\pi i \theta+\frac{z}{t}\right)\right)   &\mbox{ if }\Re(t/z)<0,\end{cases}\]
and appropriate limiting behavior at $t=0$ and $t=\infty$. We do not know whether  this problem has an interesting solution.
\end{remark}


\section{Two limits}
\label{sec_limit}

In this section we consider two limits of the solutions to the quantum Riemann-Hilbert problem discussed in the last section corresponding to $\tau\to 0$ and $\tau\to 1$ respectively. We explain how these  relate to the classical Riemann-Hilbert problem  studied in \cite{RHuncoupled,RHDT}.

\subsection{Classical Riemann-Hilbert problem}

We refer to \cite{RHDT} for the definition of a BPS structure and its associated Riemann-Hilbert problem. In this section we will use the term classical when referring to these concepts, to differentiate them from the refined BPS structures and quantum Riemann-Hilbert problems considered  above. Note that a refined BPS structure $(\Gamma,Z,\Omega)$ has an associated classical BPS structure, as defined in \cite{RHDT}, obtained by evaluating the Laurent polynomials $\Omega(\gamma)\in \bQ[\bL^{\pm \half}]$ at  the point $\bL^{\half}=-1$.

In heuristic terms, the classical Riemann-Hilbert problem associated to a classical BPS structure involves piecewise holomorphic maps
\begin{equation*}
\Phi\colon \bC^*\to \Aut(\bT_-),\end{equation*}
with prescribed limiting behaviour at $t=0$ and $t=\infty$ analogous to those imposed in Problem \ref{rough}, and discontinuous jumps
\[\Phi(t)\mapsto \Phi(t)\circ \bS(\ell),\]
given by wall-crossing automorphisms
\[\bS(\ell)^*(x_\beta)= \prod_{Z(\gamma)\in \ell} (1-x_\gamma)^{\<\gamma,\beta\>\cdot \Omega(\gamma)}\cdot x_\beta.\]
Here $\bT_-$ denotes the twisted torus defined in \eqref{oldy}, and the functions $x_\gamma\colon \bT_-\to \bC^*$ are the  twisted characters \cite[Section 2.4]{RHDT}. It is often convenient to consider the equivalent maps
\begin{equation*}
\Psi \colon \bC^*\to \Aut(\bT_-),\qquad \Phi(t)= \Psi(t)\circ \epsilon_Z(t),\end{equation*}
where $\epsilon_Z(t)$ is the translation of $\bT_-$ defined by \[\epsilon_Z(t)^*(x_\gamma)=\exp(Z(\gamma)/t)\cdot x_\gamma.\]

A quadratic refinement $\sigma\colon \Gamma\to \{\pm 1\}$ of the form $\<-,-\>$ as in Section \ref{quad}, determines an isomorphism  between the  twisted torus $\bT_-$ and the genuine torus $\bT_+$
\begin{equation}
\label{idid}\rho_\sigma\colon \bT_+\to \bT_-, \qquad \rho_{\sigma}^*(x_\gamma)=y_\gamma,\end{equation}
where $y_\gamma\colon \bT_+\to \bC^*$ denotes the genuine character of $\bT_+$ corresponding to $\gamma\in \Gamma$.  Under this identification we obtain a Riemann-Hilbert problem taking values in $\Aut(\bT_+)$, in which the wall-crossing automorphisms take the form 
\[\bS(\ell)^*(y_\beta)=\prod_{Z(\gamma)\in \ell} (1-\sigma(\gamma)\cdot y_\gamma)^{\<\gamma,\beta\>\cdot \Omega(\gamma)}\cdot y_\beta.\]
 Abusing notation, we shall use the same symbols $\Phi(t), \Psi(t)$ for the resulting piecewise holomorphic functions, now taking values in $\Aut(\bT_+)$.

Let us now consider the case of the doubled A$_1$ refined BPS structure of Example \ref{mineral}. The polynomials $\Omega(\gamma)$ are all constant, and the resulting classical BPS structure, and its associated Riemann-Hilbert problem, are precisely the ones studied  in \cite[Section 5]{RHDT} and \cite{RHuncoupled}.
We shall use the identification \eqref{idid} corresponding to the canonical quadratic refinement $\sigma\in \bT_-$ of \eqref{standard}. As usual, since there are only two active rays, the piecewise holomorphic maps $\Psi(t)$ give two functions
\[\Psi_\pm(t)\colon \bC^*\setminus i\ell_\pm \to \Aut(\bT_+).\]
It is shown in \cite{RHuncoupled} that a possible solution, which is in some sense minimal, is given by
\begin{equation}
	\label{cl4}\Psi_\pm(t)^*(y_{\alpha})=y_\alpha, \qquad \Psi_\pm(t)^*(y_{\alpha^\vee})= \Lambda\left(\pm \frac{z}{2\pi i t}, \frac{1}{2}\mp \theta\wb1\right)^{\pm 1}\cdot y_{\alpha^\vee},\end{equation}
where we have written $y_\alpha=\exp(2\pi i \theta)$. This corresponds to $x(\alpha)=\exp(2\pi i \vartheta)$, with $\vartheta=\theta+\half$. Note that since the solution \eqref{cl4} depends explicitly on $\theta$ rather than its exponential, it is not single-valued on the torus $\bT_+$.

\subsection{Limit $\tau\to 0$}
\label{abab}
The first limit  consists of sending $\tau\to 0$ and therefore $q^{\half}\to 1$. The product on the quantum torus $\bC_q[\bT]$ can be expanded around $\tau=0$
\[y_{\gamma_1}* y_{\gamma_2} = y_{\gamma_1+\gamma_2} + \pi i \tau\cdot
\{y_{\gamma_1},y_{\gamma_2}\} + O(\tau^2), \]
where $\{-,-\}$ is a  Poisson bracket on the algebraic torus $\bT_+$ given explicitly by
\[\{y_{\gamma_1},y_{\gamma_2}\}=\<\gamma_1,\gamma_2\> \cdot y_{\gamma_1+\gamma_2}.\]
As $\tau\to 0$   the solution to the quantum Riemann-Hilbert problem specified in Theorem \ref{prop_pa} becomes the solution \eqref{cl4} to the corresponding classical problem. The wall-crossing automorphisms of Lemma \ref{oneone} become
\[\bS(\ell_{\pm} )\colon y_{\alpha^\vee}\mapsto (1+y_{\pm \alpha})^{\mp 1}\cdot y_{\alpha^\vee}.\]
The adjoint description \eqref{adjy} becomes the statement that  the partially-defined automorphism $\Psi_\pm(t)$  is the time 1 Hamiltonian flow of the function
\begin{equation}
\label{finn3}H_\pm (z,t,\theta)=\lim_{\tau\to 0} \Big((2\pi i \tau)\cdot \log \psi_{\pm}(t)\Big).\end{equation}
This boils down to the statement that
\begin{equation}
	\label{hold}\frac{\partial }{\partial \theta} H_\pm (z,t,\theta)= \mp (2\pi i)\cdot \log \Lambda\left(\pm \frac{z}{2\pi i t}, \frac{1}{2}\mp \theta\b1\right).\end{equation}

Define a function
\begin{equation}
\label{back2}\Delta(w,\eta)=\frac{e^{-\zeta'(-1)}\cdot G(w+\eta+1)\cdot e^{-\frac{w^2}{4}+\frac{\eta^2}{2}-\frac{\eta}{2}+\frac{1}{12}}}{\Gamma(w+\eta)^{w+\eta}\cdot w^{-\frac{(w+\eta)^2}{2}+\frac{w+\eta}{2}-\frac{1}{12}}},\end{equation}
where $G(x)$ denotes the Barnes $G$-function, and 
\begin{equation}
\label{blblbl}\zeta'(-1)=\frac{\partial}{\partial s} \zeta_1(s,1)\big|_{s=-1},\end{equation}
is the derivative of the Riemann zeta function at $s=-1$. We refer the reader to \cite{Barnes} and \cite[Appendix]{Vo} for basic properties of the $G$-function.

\begin{lemma}
There is an expression
\[H_\pm (z,t,\theta)=-(2\pi i )\cdot \log \Delta\left(\pm \frac{z}{2\pi i t},\frac{1}{2}\mp \theta\right).\]\end{lemma}

\begin{proof}
Start with
a result of Spreafico \cite[Cor.\ 9.4]{Spreaf}, which states that
\[\lim_{\tau \to 0 } \tau\cdot \log \Gamma_2(x\b 1,\tau) =-\zeta_H(-1,x)-\frac{\partial}{\partial s} \zeta_H(s,x)\b_{s=-1} ,\]
where $\zeta_H(s,x)=\zeta_1(s,x\b 1)$ denotes the Hurwitz zeta function. The arguments of \cite[p.\ 499]{Vardi} gives the relation
\[\zeta_2(s,x\b 1,1)=\zeta_1(s-1\b 1) + (1-x)\zeta_1(s,x\b 1),\]
that can be differentiated with respect to $s$ to get 
\[-\frac{\partial}{\partial s}\zeta_1(s,x\b 1)|_{s=-1} = -\log\Gamma_2(x\b 1,1) + (1-x)\log\Gamma_1(x\b 1).\]
It is shown in \cite[Section 27]{Barnes3} that there is an identity
\begin{equation}\label{rho}\Gamma_2(x\b 1,1)^{-1}=\rho\cdot G(x)\cdot (2\pi)^{-\frac{x}{2}}=\rho \cdot G(x+1)\cdot \Gamma(x) \cdot (2\pi)^{-\frac{x}{2}}.\end{equation}
The constant $\rho=\rho(1,1)$ is inserted to allow for the fact that Barnes uses a different convention for his double gamma function which differs from the modern one by a constant  (see  \cite[Equation (3.19)]{Ruj2}). We can determine this constant by comparing the constant terms in the large $w$ asymptotic expansions of the two sides, which can be found in Corollary \ref{corGamma2exp}, and \cite[Section 15]{Barnes} or \cite[Appendix]{Vo} respectively. Using the identity \cite[(A.11)]{Vo} we obtain  $\rho=\exp(-\zeta'(1))\cdot\sqrt{2\pi}$.

Use the fact that
\[\zeta_1(-1,x\b 1)= -\half B_{1,2}(x\b1)=-\half\left(x^2-x+\frac{1}{6}\right)\]
which can be found in \cite[Appendix A]{JM} or in \cite[Section 1]{Spreaf}, to get
\[\lim_{\tau \to 0 } \tau\cdot \log \Gamma_2(x\b 1,\tau)=\frac{1}{2}\left(x^2-x+\frac{1}{6}\right)+\log G(x+1)-x\log \Gamma(x)-\log \zeta'(-1).\]
Using the definition \eqref{christmas} this easily implies that
\[\lim_{\tau \to 0 } \tau\cdot \log F(w,\eta\b 1,\tau)=\log \Delta(w,\eta).\]
The result then follows from the defintions \eqref{easter} and \eqref{finn3}.
\end{proof}

The relation \eqref{hold} follows from the identity
\[\frac{\partial }{\partial \eta} \log \Delta(w,\eta)= -\log \Lambda(w,\eta\b 1),\]
which follows from the definitions \eqref{dall} and \eqref{back2}, together with the relation
\begin{equation*}
\frac{\partial }{\partial \eta} \log G(w+\eta+1)= \frac{1}{2} - (w+\eta) +\frac{1}{2} \log(2\pi)+(w+\eta)\frac{\partial }{\partial \eta}\log \Gamma(w+\eta),\end{equation*}
which can be found  in \cite[Section 12]{Barnes} (see also \cite[Formula (A.13)]{Vo}).

\subsection{Limit $\tau\to  1$}
The second limit consists of sending $\tau\to 1$ and hence $q^{\half}\to -1$. Although the quantum torus algebra $\bC_q[\bT]$ becomes commutative in this limit, the extension \eqref{ex}--\eqref{produ} does not. Define a function
\begin{equation}
\label{upsy}\Upsilon(w,\theta)= \frac{e^{-\zeta'(-1)} \cdot G(w+\theta+1)\cdot e^{\frac{3w^2}{4}+\theta w}}{(2\pi)^{\half{(w+\theta)}}\cdot w^{\frac{1}{2}(w+\theta)^2-\frac{1}{12}}},\end{equation}
where $G(x)$ is again the Barnes $G$-function, and $\zeta'(-1)$ is given by \eqref{blblbl} as before. 

\begin{lemma}
There is an identity
\[\lim_{\tau\to 1} \psi_\pm(t) = F\Big(\pm \frac{z}{2\pi i t} ,1\mp \theta\b 1,1\Big)^{-1}= \left( \frac{\pm z}{2\pi i t}\right)^{-\frac{1}{12}}\cdot \Upsilon\left(\frac{\pm z}{2\pi i t},\mp \theta\right).\]
\end{lemma}

\begin{proof} This is an easy calculation, once the constant $\rho$ in \eqref{rho} is computed.
\end{proof}

Note that the difference relation Prop. \ref{mod} (c)  gives in the limit $\tau=1$
\begin{equation}
\label{finn}
\frac{\Upsilon(w,\theta)}{\Upsilon(w,\theta-1)}=\Lambda(w,\theta\b 1).\end{equation}

One rather mysterious feature of \cite{RHDT} was the introduction of the $\tau$-function. That paper dealt only with the case $\xi(\alpha)=1$ corresponding to $\theta=\half$.
It was proved in \cite{RHDT}  that a possible choice for the $\tau$-function in the doubled A$_1$ case is
\[\tau_\pm (z,t)= \Upsilon\left(\frac{\pm z}{2\pi i t},0\right).\]
We can therefore view the function
\begin{equation}
	\Upsilon\left(\frac{\pm z}{2\pi i t},\mp \theta\right)=  \left( \frac{\pm z}{2\pi i t}\right)^{\frac{1}{12}}\cdot \lim_{\tau\to 1} \psi_\pm(t)\end{equation}
as an extension of the function $\tau_\pm$ to all values of $\theta$. Note however that there is a confusing shift  here: with our conventions the classical Riemann-Hilbert problem studied in \cite[Section 5.3]{RHDT} corresponds to $\theta=\half$.
The  difference relation \eqref{finn} implies that\[\Upsilon\left(\pm \frac{ z}{2\pi i t},\mp\left(\theta+\half\right)\right)=\Upsilon\left(\pm \frac{ z}{2\pi i t}\mp\left(\theta-\half\right)\right)\cdot\Lambda\left(\pm \frac{z}{2\pi i t}, \half\mp \theta\b1\right)^{\pm 1},\]
which perhaps gives a clue as to the true nature of the $\tau$-function.


\section{The general case}
\label{general}
In this section we consider the general quantum Riemann-Hilbert problem corresponding to a refined BPS structure satisfying the  four conditions of Definition \ref{leaf}. We shall be rather brief since the proofs are all identical to the ones for the doubled A$_1$ case, just with added notation.\footnote{The correction indicated in footnote \ref{change}  requires  some changes to the formulae of Sections 5.2 and 5.3   compared with the published version. We also corrected a couple of more minor mistakes.}

\subsection{Extended quantum torus}

Let us consider a refined BPS structure satisfying the four conditions of Definition \ref{leaf}. As before we will use a quadratic refinement of the form $\<-,-\>$ on $\Gamma$ to introduce some convenient signs. For simplicity we assume that we can find such a quadratic refinement $\sigma\colon \Gamma\to \{\pm 1\}$ with the property that
\begin{equation}
\label{grrr}\Omega_n(\gamma)\neq 0 \implies \sigma(\gamma)=(-1)^{n+1}.\end{equation}
The general palindromic case is  analogous with some terms occurring with a different sign.

As in Section \ref{quad} we then introduce alternative generators for $\bC_q[\bT]$ 
\begin{equation}
\label{hi}q^{\half}=-\bL^{\half}, \qquad y_\gamma=\sigma(\gamma)\cdot x_\gamma.\end{equation}

The uncoupled assumption ensures that we can decompose
\begin{equation}
\label{decompose}\Gamma=\Gamma_e\oplus \Gamma_m,\end{equation}
in such a way that
$\Omega(\gamma)=0$ unless $\gamma\in \Gamma_e$, and the form $\<-,-\>$ vanishes when restricted to $\Gamma_e$ and $\Gamma_m$ separately. Introduce the vector space
\[V_e=\Hom_\bZ(\Gamma_e,\bC)\isom \bC^k,\]
and denote a typical element by $\theta\colon \Gamma_e\to \bC$. Note that each  element $\delta\in \Gamma_m$ determines a corresponding element  $\<\delta,-\>\in V_e$.

We can write down an extended quantum torus algebra
\begin{equation}
\label{grad}\widehat{\bC_q[\bT]}=\bigoplus_{\delta\in \Gamma_m} \Mer(\cH\times V_e)\cdot y_{\delta},\end{equation}
in much the same way as before, where the coefficients of the formal symbols $y_\delta$ are meromorphic functions $f(\tau,\theta)$ on the product of the upper half-plane $\cH$ with the vector space $V_e$. The product is
\[\Big(f_1(\tau,\theta)\cdot y_{\delta_1}\Big) *\Big( f_2(\tau,\theta)\cdot y_{\delta_2}\Big)={ f_1(\tau,\theta)\cdot f_2(\tau,\theta+\tau\<\delta_1,-\>)}\cdot y_{\delta_1+\delta_2},\]
and there is  an injective homomorphism
$I\colon \bC_q[\bT]\into \widehat{\bC_q[\bT]}$ defined by \[I\colon q^{\frac{k}{2}} \cdot y_{\gamma_e+\gamma_m}\mapsto \exp\big( \pi i  (k+\<\gamma_m,\gamma_e\>)\tau+2\pi i \theta(\gamma_e)\big)\cdot y_{\gamma_m},\]
where $(\gamma_e,\gamma_m)$ denotes an arbitrary element of $\Gamma$ under the decomposition \eqref{decompose}.
As before we identify elements of $\bC_q[\bT]$ with their images under the embedding $I$. Note that
\begin{equation}
\label{hh}I(q^{\half})=\exp(\pi i \tau)\cdot 1, \qquad I(y_{\gamma_e})=\exp(2\pi i \theta(\gamma_e))\cdot 1, \qquad I(y_{\gamma_m})=y_{\gamma_m},\end{equation}
for $\gamma_e\in \Gamma_e$ and $\gamma_m\in \Gamma_m$. 

\subsection{Automorphisms associated to rays}

According to \eqref{tir}, and using \eqref{grrr} and \eqref{hi},  the  automorphism associated to an active ray $\ell\subset \bC^*$ is
\begin{equation*}
	\bS_q(\ell)=\Ad_{\DT_q(\ell)}, \qquad \DT_q(\ell)=\prod_{Z(\gamma)\in \ell} \prod_{n\in \bZ} \bE_q\left((- q^{\half})^{n+1}\cdot y_\gamma\right)^{(-1)^{n-1}\, \Omega_n(\gamma)},\end{equation*}
which makes sense in the extended quantum torus as before.
To give a formula for it we first introduce some notation. Given classes $\beta,\gamma\in \Gamma$, let \[\epsilon(\beta,\gamma)\in \{\pm 1\}\] denote the sign of $\<\beta,\gamma\>$, and \[\kappa(\beta,\gamma)=\Big\{\lambda =\epsilon(\beta,\gamma)\cdot (2j+1) : 0\leq  j< |\<\beta,\gamma\>|\Big\},\]
denote the set of odd integers  lying between $0$ and $2\<\beta,\gamma\>$.
\begin{prop}
\label{newone}
The automorphism $\bS_q(\ell)$ of the algebra  \eqref{grad} preserves the grading,  acts trivially on the zeroth graded piece, and satisfies
\begin{equation}
\label{autaut}
\bS_q(\ell)(y_\beta)= \prod_{Z(\gamma)\in \ell}  \prod_{\lambda\in \kappa(\beta,\gamma)}\prod_{n\in \bZ} \left( 1- (-q^\half)^{\lambda+n} \cdot y_\gamma\right)^{(-1)^{n-1}\cdot \Omega_n(\gamma)\cdot \epsilon(\beta,\gamma)}\cdot y_\beta,
\end{equation}
for any class $\beta\in \Gamma_m$. 
\end{prop}

\begin{proof}
This follows by an explicit computation exactly as in Lemma \ref{oneone}. 
\end{proof}

\subsection{Solution in general case}

As in \cite[Section 4]{RHDT} it is best to consider the solution to the Riemann-Hilbert problem to be a collection of functions
\[\Phi_r\colon \bH_r\to \Aut\widehat{\bC_q[\bT]},\]
 defined on each half-plane $\bH_r$ centered on a non-active ray $r\subset \bC^*$. As before we write $\Psi_r(t)=\Phi_r(t)\circ \epsilon_Z(t)$. 
 
 \begin{theorem}
A solution to the quantum Riemann-Hilbert problem in the case of a refined BPS structure satisfying the four conditions of Definition \ref{leaf} is given by the collection of functions
\[\Psi_r(t)( y_\beta)= \prod_{Z(\gamma)\in i\bH_r} \prod_{\lambda\in \kappa(\beta,\gamma)}\prod_{n\in \bZ} \Lambda\left( \frac{Z(\gamma)}{2\pi i t},\frac{1+n}{2}- \theta(\gamma)- \half (\lambda+n)\tau\wb 1\right)^{(-1)^n\cdot \Omega_n(\gamma)\cdot \epsilon(\beta,\gamma)} \cdot y_\beta , \]
where the outer product is over the finitely many active classes $\gamma\in \Gamma_e$ for which  $Z(\gamma)\in i\bH_r$. 
\end{theorem}

\begin{proof}
Consider small clockwise (respectively anti-clockwise) perturbations $r_+$ (respectively $r_-$) of an active ray $\ell$. Note that the product in the statement of the Theorem for the rays $r_\pm$ differ precisely by products over the classes $\gamma\in \Gamma_e$ satisfying $Z(\gamma)\in\pm \ell$. It follows that for $t\in \bH_\ell$
\[\Psi_{r_+}(t)( y_\beta)=\prod_{Z(\gamma)\in \ell} \prod_{\lambda\in \kappa(\beta,\gamma)}\prod_{n\in \bZ}  \left( 1+(-1)^n e^{ 2\pi i \theta(\gamma)+\pi i (\lambda+n)\tau-\frac{Z(\gamma)}{t} }\right)^{(-1)^{n-1}\cdot \Omega_n(\gamma)\cdot \epsilon(\beta,\gamma)}\cdot \Psi_{r_-}(t) (y_\beta),\]
where we used the palindromic assumption $\Omega_n(\gamma)=\Omega_{-n}(\gamma)$ together with Proposition \ref{lemGamma}(c). Using the formula \eqref{autaut} and the identifications \eqref{hh} this agrees with the wall-crossing automorphism
\[\tilde{\bS}_q(\ell)= \epsilon_Z(-t)\circ {\bS_q}(\ell) \circ \epsilon_Z(t)\in \Aut \widehat{\bC_q[\bT]}.\]
The other conditions of the  Riemann-Hilbert problem are checked in exactly the same way as before, since the product appearing in the statement of the Theorem is finite.
\end{proof}

The adjoint form is given by the expression
\[\psi_{r}(t)=\prod_{Z(\gamma)\in i\bH_r}\prod_{n\in \bZ} F\bigg(\frac{Z(\gamma)}{2\pi i t},\frac{1+n}{2} +\frac{(1-n)\tau}{2}- \theta(\gamma)\wb1, \tau\bigg)^{(-1)^{n-1}\cdot \Omega_n(\gamma)}.\]
We recover the formulae of Section \ref{sec_soln} by gluing the solutions $\Psi_r(t)$ for rays $r$ contained in the half-plane $\pm \Im(t/z)<0$ to obtain a function $\Psi_\pm(t)$ on the domain $\bC^*\setminus i\ell_\pm$.


\appendix

\section{A second Stirling formula for the multiple Gamma functions}
\label{app}

The goal of this Appendix is to compute an asymptotic expansion for $\log\Gamma_N(x+\delta\b\underline{a})$, $N\geq 1$, analogous to the second Stirling approximation for the Gamma function. We assume once for all that $\underline{a}\in \left(\bC^*\right)^N$ with $a_i$ lying in the same open half-plane in $\bC^*$, and denote by $\lambda_{\underline{a}}$ a non-zero complex number such that $\Re(\lambda_{\underline{a}} \cdot a_i)>0$ for every $i=1,\dots,N$. 
\begin{theorem}\label{stirling2}
Let $x,\delta\in\bC$, with $|\arg\left(\frac{x}{\lambda_{\underline{a}}}\right)|<\pi$ and $|\arg\left(\frac{x+\delta}{\lambda_{\underline{a}}}\right)|<\pi$, then
 $$\log \Gamma_N\left(x+\delta \b \underline{a}\right)$$ has asymptotic expansion as $|x|\to \infty$ away from poles
	\begin{gather*}
	\frac{(-1)^{N+1}}{N!} \Big( B_{N,N}(x+\delta\b \underline{a})\log(x) - \sum_{k=0}^{N-1}c_{N,k}\,B_{N,k}(0\b\underline{a})(x+\delta)^{N-k} + P_{N-1}(x,\delta\b \underline{a}) \Big) \\
	+\sum_{k>0}\frac{(-1)^{N+k}\cdot B_{N,N+k}(\delta\b\underline{a})}{k(k+1)\cdots(k+N)} \cdot x^{-k},
	\end{gather*}
where 
\begin{enumerate}
\item $c_{N,k}=\binom{N}{k}\cdot\sum_{l=1}^{N-k}l^{-1}$ are combinatorial factors, and
\item $P_{N-1}(x,\delta\b \underline{a})$ is a polynomial of degree $N-1$ consisting in the non-negative degree terms of the Laurent polynomial
	\[ B_{N,N}(x+\delta\b \underline{a})\sum_{n=1}^N\frac{(-1)^{n+1}\delta^n}{n}x^{-n} .
	\]
\end{enumerate}
\end{theorem}
When $N=1$, the formula above recovers the usual second Stirling expansion for the Gamma function, recalling that $\Gamma_1(x\b a)=\Gamma(x/a)\cdot a^{\frac{x}{a}-\frac{1}{2}}\cdot (2\pi)^{-\half}$. We are particularly interested in the case $N=2$.

\begin{cor}\label{corGamma2exp} Let $a_1,a_2$ be two non-zero complex numbers lying in the same half-plane, $\lambda=\lambda_{(a_1,a_2)}$, $x,\delta\in\bC$. Then
	\begin{equation*}\label{ggg}\begin{aligned}
	\log \Gamma_2\left(x+\delta \b a_1,a_2\right) \sim & -\frac{1}{2} B_{2,2}(x+\delta\b a_1,a_2)\log x + \frac{3x^2}{4 a_1 a_2} - \frac{x(a_1+a_2)}{2 a_1 a_2} + \frac{\delta x}{a_1 a_2} + \\
	&+\sum_{k>0}\frac{(-1)^k\cdot B_{2,k+2}(\delta\b a_1,a_2)}{k(k+1)(k+2)} \cdot x^{-k}
	\end{aligned}\end{equation*}
is valid for $|x|\to\infty$ away from poles as long as $|\arg(x/\lambda)|, |\arg(x+\delta)/\lambda| <\pi$.
\end{cor}
\begin{proof} It is an application of the Theorem above when $N=2$. In particular we have
	\[ B_{2,2}(x+\delta\b a_1,a_2)\cdot \left(\frac{\delta}{x}-\frac{\delta^2}{2 x^2}\right) = \frac{x\delta}{a_1a_2} + \frac{2\delta^2}{a_1a_2} -\frac{\delta^2}{2a_1a_2} - \delta\left(\frac{1}{a_1}+\frac{1}{a_2}\right) + O(x^{-1}).\]
We also have $B_{2,0}(0\b a_1,a_2)=\frac{1}{a_1a_2}$, $B_{2,1}(0\b a_1,a_2)=-\frac{a_1+a_2}{2a_1a_2}$, and $c_{2,0}=1\cdot\frac{3}{2}$, $c_{2,1}=1$. 
\end{proof}

The proof of Theorem \ref{stirling2} given below is mostly a rephrasing of the proof of the asymptotic expansion of $\Gamma(x+\delta)$ that can be found in \cite[{Sec.\ 13.6}]{WW}. It is based on the comparison with the standard asymptotic expansion of $\log\Gamma_N(y\b \underline{a})$ when $|y|\to\infty$, $y\not\in\bR_{<0}$, \cite[{Eq.\ 3.13}]{Ruj2}
	\begin{equation}\label{asy}\begin{split}
	\frac{(-1)^{N+1}}{N!}B_{N,N}(y\b \underline{a})\log(y) + (-1)^N\sum_{k=0}^{N-1} \frac{B_{N,k}(0\b\underline{a})y^{N-k}}{k!(N-k)!}\sum_{l=1}^{N-k}l^{-1} + \\
	+ \sum_{k\geq N+1} (-1)^k\frac{(k-N-1)!}{k!}B_{N,k}(0)(y)^{N-k},
	\end{split}\end{equation}
and depends on the next results. We denote by
	\begin{equation*}
	\mathcal{W}\left(x,\underline{a}\right):=\prod_{\underline{n}\in\bN^N\setminus\{0\}}\left(1+\frac{x}{\underline{n}\cdot\underline{a}}\right)\cdot \exp\left(\sum_{j=1}^N\frac{(-1)^j}{j}\frac{x^j}{(\underline{n}\cdot\underline{a})^j}\right)
	\end{equation*}
the canonical Weierstrass product associated with $\zeta_N(s,x\b \underline{a})$. It is uniformly and absolutely convergent in any bounded closed region of the complex plane
, meaning that the corresponding logarithm series converges uniformly and absolutely. An application of the \lq\lq Lerch formula\rq\rq\ by Spreafico, \cite[{Prop.\ 2.9}]{Spreaf2}, shows that
	\begin{equation}\label{lerch_prop} 
	\log \Gamma_N\left(x\b \underline{a}\right) 
	= \Gamma_N(0\b\underline a) - \log \mathcal{W}\left(x,\underline{a}\right) - \log x + q_N(x\b\underline{a}),
	\end{equation}
	for a polynomial $q_N(x\b\underline{a})$ of degree $N$. The polynomial $q_N(x\b\underline{a})$ is explicitly given in \cite{Spreaf2} in terms of the residues of $\zeta_N(s,0\b\underline{a})$.\footnote{Note that in Spreafico's notation $\zeta(s,S_x)$ and $F(x,S_0)$ correspond respectively to our $ x^{-s}\cdot\zeta_N\left(s,x\b \underline{a}\right)$ and $\mathcal{W}(x,\underline{a})$, assuming $S_x$ to be the sequence $S_x=\left(\underline{n}\cdot\underline{a}+x\right)_{\underline{n}\in\bN^N\setminus\{0\}}$.}
For fixed $\delta\in\bC$, $x\in\bC^*$, with $|\arg x|<\pi$, we introduce the function $$g(s) := \frac{\pi x^s}{s\cdot \sin(\pi s)}\zeta_N(s,\delta\b \underline{a}).$$
\begin{lemma}\label{lem1}
$g(s)$ has poles at $s\in\bZ$ whose residues are well-defined functions in $x$. In particular
\begin{align*}
\text{if }s&=k\in\bZ\setminus\{0,\dots,N\} & Res(g,k) &= \frac{(-1)^k\cdot x^k}{k} \zeta_N\left(k,\delta\b \underline{a}\right),\\
\text{if }s&=0 & Res(g,0) &= \zeta_N(0,\delta\b \underline{a})\cdot \log x + \log \Gamma_N(\delta\b \underline{a}),
\end{align*}
while for $s=j\in\{1,\dots,N\}$, $Res(g,j)$ is of type $c_j\log(x)+d_j$, for some $c_j,d_j\in\bC$. 
\end{lemma}
\begin{proof} We expand in series around $s=k\in\bZ$
	\begin{equation*}\begin{aligned}
	x^s&=x^k \cdot \sum_{n\geq 0} \frac{\big((s-k)\log x\big)^n}{n!} = x^k \cdot \big( 1+(s-k)\log x+O(s-k)^2\big) \\
	\frac{1}{s} &= \frac{1}{k} \cdot \frac{1}{1-\left(\frac{k-s}{k}\right)} = \frac{1}{k} \sum_{n\geq 0} \left(\frac{-(s-k)}{k}\right)^n = \frac{1}{k}-\frac{1}{k^2}(s-k) + O(s-k)^2\\
	\sin(\pi s)^{-1} &= (-1)^k\big(\sin \pi(s-k)\big)^{-1} = (-1)^k \left(\frac{1}{\pi(s-k)} + \frac{1}{6}\pi(s-k) + O(s-k)^3\right).
	\end{aligned}\end{equation*}
Around $s=0$, the Taylor series of $\zeta_N(s,\delta\b \underline{a})$ is $\zeta_N(0,\delta\b \underline{a})+ s \log\Gamma_N(\delta\b\underline a) + O(s^2)$. The multiple zeta function $\zeta_N(k,\delta\b \underline{a})$ has poles at $k=1,\dots,N$, around which it can be written as 
	\[
	\zeta_N\left(s,x\b \underline{a}\right) = R^{-1}_N(k,x\b\underline a)\cdot (s-k)^{-1} + R_N^0(k,x\b\underline a) + O(s-k).\]
$Res(g,j)$ is therefore given by $\left(\log x -\frac{1}{j} \right) R_N^{-1}(k,\delta\b\underline a) + R_N^0(k,\delta\b\underline a)$.
\end{proof}
\begin{lemma}\label{lem2} Assume $\Re x >0$, $\Re a_i>0$, $|a_i|\leq 1$ for all $i$. In the following cases
	\begin{itemize}
	\item[a)]
	$|x|<1$, and $\cC$ is an arc of large radius contained in $\Re s >N$ and centered on $\bar{s}\in\bR$, $N-1<\bar{s}<N$,
	\item[b)] 
	$\cC=I\times i R$, $I\subset\{\Re s<N+1\}$ a closed real interval, and $R\gg 0$,
	\end{itemize}
	the integral $\int_{\cC}g(s)ds$ vanishes.
	\end{lemma}
	\begin{proof}
	a) follows from the fact that $\zeta_N(s,\delta\b\underline{a})\to 0$ for $\Re s> N$, $|s|\gg 0$, \cite[{Eq.\ 3.8}]{Ruj2}. For b) we prove that the integrand is dominated by $e^{-|\Im s|}$ when $\Re s$ is bounded above and below and $|\Im s| \gg 0$. By Theorem 3 in \cite{matsumoto} we have that, for $s\in\bC\setminus\{1,\dots,N\}$,	
	\begin{equation*}
	\zeta_N(s,\delta\b\underline{a})= \left(\frac{1}{s-1}+\frac{1}{2}\right)\zeta_{N-1}(s-1,\delta\b (a_1,\dots,a_{N-1})) + O(1).
	\end{equation*}
We know that when $\Re s\leq 1$, $\zeta(s,\delta) = O\big(|\Im s|^{1-\Re s}\cdot\log|\Im s|\big)$, \cite[{Sec.\ 13.5}]{WW}. Then proceeding inductively we obtain a bounded behaviour for $\zeta_N(s,\delta\b \underline a)$. For $|\Im s| \gg 0$ the integrand is dominated by $\frac{1}{s\sin s}\sim e^{-|\Im s|}$. The conclusion follows.
\end{proof}

\begin{proof}[{Proof of Theorem \ref{stirling2}}]
We first observe that, since $\zeta_N'(s,x\b x,\underline{a})_{|s=0} = \frac{(-1)^N}{N!} B_{N,N}(x\b\underline{a})$, the following holds
	\[\log \Gamma_N(\lambda x\b \lambda \underline{a})= \frac{(-1)^{N+1}}{N!}B_{N,N}(x\b \underline{a})\log \lambda +\log \Gamma_N(x\b \underline{a}).
	\]
This, together with the homogeneity property \eqref{homoBnk}, implies that %
it is enough to prove the statement for $\Re(a_i)\geq 0$, $|a_i|\leq 1$.
	
Assume initially that $\Re(x)>0$. Consider the difference $\log \Gamma_N(\delta\b \underline{a})-\log\Gamma_N(x+\delta\b \underline{a})$. By \eqref{lerch_prop} it equals
	\begin{equation}\label{1}\begin{split}
	q_N(x\b\underline{a})-q_N(x+\delta\b\underline{a}) + \log \frac{\mathcal{W}(x+\delta,\,\underline{a})}{\mathcal{W}(\delta,\,\underline{a})} + \log\left(\frac{x+\delta}{\delta}\right).
	\end{split}\end{equation}
$\log \frac{\mathcal{W}(x+\delta,\, \underline a)}{\mathcal{W}(\delta,\, \underline a)}$ is the absolutely convergent series 
	\begin{gather*}
	\sum_{\underline{n}\in\bN^N\setminus\{0\}} \left[\log\left(1+\frac{x}{\underline{n}\cdot\underline{a}+\delta}\right) + \sum_{j=1}^N \frac{(-1)^j}{j}\left(\frac{(x+\delta)^j}{(\underline{n}\cdot \underline{a})^j}- \frac{\delta^j}{(\underline{n}\cdot \underline{a})^j}\right)\right] .
	\end{gather*}
For $|x|<\min\{1, |\delta|, |a_i|\b i=1,\dots N\}$ the logarithms $\log\left(1+\frac{x}{\underline{n}\cdot\underline{a}+\delta}\right)$ and $\log\left(\frac{x+\delta}{\delta}\right)$ can be expanded in (absolutely convergent) series and the second half of \eqref{1} reads
	\begin{equation}\label{2}\begin{split}
		\sum_{\underline{n}\in\bN^N\setminus\{0\}}
		\left( \sum_{k=1}^N\frac{(-1)^{k-1}x^k}{k}  \frac{1}{(\underline n \cdot \underline a + \delta)^k} + 
		\sum_{k>N}\frac{(-1)^{k-1}x^k}{k}  \frac{1}{(\underline n \cdot \underline a + \delta)^k} 	\right) + \\
		\sum_{\underline{n}\in\bN^N\setminus\{0\}}  
		\sum_{j=1}^N \frac{(-1)^j}{j}\left(\frac{(x+\delta)^j}{(\underline{n}\cdot \underline{a})^j}- \frac{\delta^j}{(\underline{n}\cdot \underline{a})^j}\right) + 
		\sum_{k=1}^N\frac{(-1)^{k-1}x^k}{k}  \frac{1}{\delta^k} + 
		\sum_{k>N}\frac{(-1)^{k-1}x^k}{k}  \frac{1}{ \delta^k} .
	\end{split}
	\end{equation}
The two sums over $k>N$ converge absolutely and can be extracted from the series. Switching the order of the sum, they give 
	\begin{equation*}
	 \sum_{k>N} \frac{(-1)^{k-1}}{k} x^k \zeta_N(k,\delta\b \underline a),
	\end{equation*}
which in turn, by Lemma \ref{lem1}, coincides with $\int_{\cC}\frac{\pi x^s}{s\cdot \sin(\pi s)}\zeta_N(s,\delta\b\underline{a})ds$, where $\cC$ is a contour encircling clockwise the integers $k>N$. We assume for a moment that $|x|<1$. Keeping in mind Lemma \ref{lem2}, we deform continuously $\cC$ to a path encircling the negative integers. Applying again Lemma \ref{lem1}, it equals
	\[
		C\log(x)+D + \zeta_N(0,\delta\b \underline a) \cdot \log x + \log \Gamma_N(\delta\b \underline a) + \sum_{k>0}\frac{(-1)^kx^{-k}}{-k}\zeta_N(-k,\delta\b \underline a),
	\]
for some complex numbers $C,D$. We recall that $\zeta_N(-k,\delta\b \underline a)= \frac{(-1)^N k!}{(N-k)!}B_{N,N-k}(\delta\b \underline{a})$ for $k\in\bN$. Summing everything together, the resulting expression for 
	\[
	\log \Gamma_N(\delta\b \underline{a}) - \log \Gamma_N(x+\delta\b \underline{a}),
	\]
for $|x|<\min\{|\delta|, |a_i|\b i=1,\dots N\}$, becomes
	\begin{gather*}
	\log \Gamma_N(\delta\b \underline a) + p(x,\delta,\underline{a}) + \sum_{k>0}\frac{(-1)^{N+k}k!}{-k(N+k)!}B_{N,N+k}(\delta\b\underline{a})x^{-k},
	\end{gather*}
where $p(x,\delta,\underline{a})$ contains polynomial and logarithmic terms and the contribution of the remaining of \eqref{2}.
By analytic continuation, the formula holds for every $x\in\bC$, $|\arg x|<\pi$, away from poles. 
Assuming also $|\arg(x+\delta)|<\pi$, we compare it with \eqref{asy} evaluated at $y=x+\delta$.
This, up to terms of order $O(x^{-1})$, can be rewritten as
	\begin{gather}
		\frac{(-1)^{N+1}}{N!}B_{N,N}(x+\delta\b \underline{a})\log(x) + (-1)^N\sum_{k=0}^{N-1} \frac{B_{N,k}(0\b\underline{a})\sum_{l=1}^{N-k}l^{-1}}{k!(N-k)!}(x+\delta)^{N-k}+\label{g1}\\
		+\frac{(-1)^{N+1}}{N!}B_{N,N}(x+\delta\b \underline{a})\log\left(1+\frac{\delta}{x}\right) \label{g2}
	\end{gather}
$B_{N,N}(x\b \underline{a})$ is a polynomial of degree $N$. The comparison shows that $p(x,\delta,\underline{a})$ must equal 
\[\eqref{g1} 
+ P_{N-1}(x,\delta\b\underline{a}),\]
where $P_{N-1}(x,\delta\b\underline{a})$ consists in the non-negative degree terms of \eqref{g2} after expanding in series $\log\left(1+\frac{\delta}{x}\right)$. 
\end{proof}


\begin{thebibliography}{100}

\bibitem{RHuncoupled} A.\ Barbieri, A Riemann-Hilbert problem for uncoupled BPS structures, manuscripta mathematica, 162(1) (2020), 1--21.


\bibitem{Barnes} E.W.\ Barnes, The theory of the G-function, Quarterly Journ. Pure and Appl. Math. 31 (1900), 264--314.

\bibitem{Barnes2} E.W.\ Barnes, The genesis of the double gamma function, Proc. London Math. Soc. 31 (1900), 358--381.

\bibitem{Barnes3} E.W.\ Barnes, The theory of the double gamma function,  Philos.\ Trans.\ Roy.\ Soc.\ A 196 (1901), 265--388.

\bibitem{RHDT} T.\ Bridgeland, Riemann-Hilbert problems from Donaldson-Thomas theory, to appear Invent.\ math.\ 216 (2019), 69-124.

\bibitem{conifold} T.\ Bridgeland, Riemann-Hilbert problems for the resolved conifold, J.\ Differential Geom.\ 115 (2020), no.\ 3, 395--435

\bibitem{CNV} S.\ Cecotti, A.\ Neitzke, C.\ Vafa, Twistorial topological strings and a $tt^*$ geometry for $\mathcal{N}=2$ theories in 4d, Adv. in Theor. and Math. Phys. 20 (2016), 193--312.


\bibitem{FS} S.\,A.\ Filippini and J.\ Stoppa, TBA equations and tropical curves, Internat.\ J.\ Math.\ 27 (2016), no.\ 7, 1640005, 19 pp.

\bibitem{FS2} S.\,A.\ Filippini and J.\ Stoppa, Block-G\"ottsche invariants from wall-crossing, Compositio Math.\ 151 (2015), no.\ 8, 1543 --1567. 


\bibitem{FG} V.\ Fock and A.\ Goncharov, The quantum dilogarithm and representations of quantum cluster varieties, Invent.\ Math.\ 175 (2009), no.\ 2, 223--286.

\bibitem{FR} E.\ Friedman and S.\ Ruijsenaars, Shintani--Barnes zeta and gamma functions, Adv.\ Math.\ 187 (2004) 362 -- 395.

\bibitem{GMN1} D.\ Gaiotto, G.\ Moore and A.\ Neitzke, Four-dimensional wall-crossing via three-dimensional field theory, Comm.\ Math.\ Phys.\ 299 (2010), no. 1, 163--224.

\bibitem{GMN2} D.\ Gaiotto, G.\ Moore and A.\ Neitzke, Wall-crossing, Hitchin systems, and the WKB approximation, Adv.\ Math.\ 234 (2013), 239--403.


\bibitem{JM} M.\ Jimbo and T.\ Miwa, Quantum KZ equation with $ \b q \b  =1$ and correlation functions of the XXZ model in the gapless regime, J.\ Phys.\ A 29 (1996), no.\ 12, 2923--2958. 

\bibitem{KS} M.\ Kontsevich, Y.\ Soibelman, Stability structures, motivic Donaldson-Thomas invariants and cluster transformations, arXiv preprint arXiv:0811.2435 (2008)

\bibitem{matsumoto} K.\ Matsumoto, The analytic continuation and the asymptotic behaviour of certain multiple zeta-functions I, Journal of Number Theory 101 (2003) 223--243.

\bibitem{ON} A.\ Okounkov and N.\ Nekrasov, Seiberg-Witten theory and random partitions. The unity of mathematics, 525--596, Progr.\ Math., 244, Birkh{\"a}user, Boston, MA, 2006.

\bibitem{Ruj2} S.N.M.\ Ruijsenaars, On Barnes Multiple Zeta and Gamma Functions, Adv.\ Math.\ 156 (2000), no. 1, 107--132.

\bibitem{Spreaf} M.\ Spreafico, On the Barnes double zeta and gamma functions, Jour.\ Number Th., 129 (2009) 2035--2063.

\bibitem{Spreaf2} M.\ Spreafico, Zeta invariants for sequences of spectral type, special functions and the Lerch formula, Proceedings of the Royal Society of Edinburgh, 136A (2006) 863--887. 

\bibitem{Vardi} A.\ Vardi, Determinants of Laplacians and multiple gamma functions, SIAM J.\ Math.\ Anal., 19 (1988), no.\ 2, 493--507.

\bibitem{Vo} A.\ Voros, Spectral Functions, Special Functions and the Selberg Zeta Function, Commun.\ Math.\ Phys.\ 110 (1987), 439--465.

\bibitem{WW} E.\,T.\ Whittaker and G.\,N.\ Watson, A course of modern analysis. An introduction to the general theory of infinite processes and of analytic functions; with an account of the principal transcendental functions. Reprint of the fourth (1927) edition. Cambridge University Press, 1996. vi+608 pp.



\end{thebibliography}
\end{document}